\documentclass[11pt, letterpaper]{amsart}

\usepackage{amsmath}
\usepackage{amssymb, mathrsfs}
\usepackage{amsfonts, mathrsfs}
\usepackage{amsthm}
\usepackage{relsize}
\usepackage{enumerate}
\usepackage{verbatim}
\usepackage{amscd}
\usepackage{geometry}
\usepackage{graphicx}
\usepackage{appendix}
\usepackage{tikz}
\usepackage{hyperref}

\newtheorem{innercustomgeneric}{\customgenericname}
\providecommand{\customgenericname}{}

\newcommand{\newcustomtheorem}[2]{\newenvironment{#1}[1]
  {\renewcommand\customgenericname{#2}
   \renewcommand\theinnercustomgeneric{##1}\innercustomgeneric}{\endinnercustomgeneric}}

\newcustomtheorem{customthm}{Theorem}

\newcommand{\newcustomlemma}[2]{\newenvironment{#1}[1]
  {\renewcommand\customgenericname{#2}
   \renewcommand\theinnercustomgeneric{##1} \innercustomgeneric}{\endinnercustomgeneric}}

\newcustomlemma{customlemma}{Lemma}

\newcustomlemma{customproposition}{Proposition}

\theoremstyle{plain}
\newtheorem{theorem}{Theorem}[section]

{}
\newtheorem{lemma}[theorem]{Lemma}
\newtheorem{proposition}[theorem]{Proposition}
\newtheorem{corollary}[theorem]{Corollary}
\theoremstyle{definition}

\theoremstyle{remark}
\newtheorem{remark}{Remark}
\newtheorem*{notation}{Notation}
\numberwithin{equation}{section}

\newcommand{\R}{\mathbb{R}}
\newcommand{\Z}{\mathbb{Z}}

\newcommand{\cM}{\mathcal{M}}

\newcommand{\q}{\quad}

\newcommand{\wh}{\widehat}
\newcommand{\supp}{\mbox{supp}}

\begin{document}

\title[Maximal operators with Fourier multipliers]
{Maximal operators associated with Fourier multipliers and applications}

\author[J. B. Lee]{Jin Bong Lee}
\address[J. B. Lee]{Research Institute of Mathematics, Seoul National University, Seoul 08826, Republic of Korea}
\email{jinblee@snu.ac.kr}

\author[J. Seo]{Jinsol Seo}
\address[J. Seo]{Department of Mathematics, Korea University, 145 Anam-ro, Seongbuk-gu, Seoul, 02841, Republic of Korea}
\email{seo9401@korea.ac.kr}

\thanks{J. B. Lee is supported by the National Research Foundation of Korea(NRF) grant 2021R1C1C2008252 and the National Research Foundation of Korea(NRF) grant 2022R1A4A1018904}


\begin{abstract}
In this paper, we introduce a criterion for maximal operators associated with Fourier multipliers to be bounded on $L^p(\R^d)$ for each $p\in(1,\infty)$.
Noteworthy examples satisfying the criterion are multipliers of the Mikhlin type or limited decay which are not necessarily radial.
To do so, we make use of modified square function estimates and bilinear interpolation. 
For the bilinear interpolation we introduce a function space $\Sigma^2(\mathcal{B})$ where $\mathcal{B}$ denotes a Banach space of functions on $\R^d$, which is a variant of weighted Sobolev spaces.
In result, we obtain convergence results for fractional half-wave equations and surface averages as well as the $L^p$ boundedness for the maximal operators.
\end{abstract}

\maketitle

\section{Introduction}

Maximal operators and Fourier multipliers are crucial objects in harmonic analysis 
  due to their importance in harmonic analysis itself 
    and vast applications in other areas of mathematics such as theory of partial differential equations. 
Maximal operators of our particular interest are given in terms of Fourier multipliers.
One of the best examples of this kind could be the spherical maximal function
\begin{align*}
  \mathcal{M}_{sph}(f)(x) &:= \sup_{t>0} \bigg| \int_{\mathbb{S}^{d-1}} f(x-ty) d\sigma_{sph}(y)\bigg|,
\end{align*}
where $d\sigma_{sph}$ denotes the spherical measure on $\mathbb{S}^{d-1}$ induced by the Lebesgue measure.
For $f\in \mathscr{S}(\R^d)$, we make use of the Fourier inversion formula to write
\begin{align*}
  \mathcal{M}_{sph}(f)(x) &:= \sup_{t>0} \bigg| \int_{\R^d} e^{2\pi i \langle x, \xi\rangle}  \wh{d\sigma_{sph}}(t\xi) \wh{f}(\xi) d\xi\bigg|,
\end{align*}
which is the form of $\sup_{t>0}\big| \big(m(t\cdot)\wh{f}(\cdot)\big)\check{\,\,} \big|$.
Note that $\wh{d\sigma_{sph}}$ is the Fourier transform of the surface-carried measure of $\mathbb{S}^{d-1}$ 
and equals to $\frac{J_{(d-2)/2}(|\xi|)}{|\xi|^{(d-2)/2}}$ up to constant, where $J_\nu(r)$ is the Bessel function of order $\nu$. 
$\mathcal{M}_{sph}$ is bounded on $L^p(\R^d)$ for $p>\frac{d}{d-1}$ when $d\geq3$ due to E. M. Stein \cite{St1976,St2016} and for $d=2$ by J. Bourgain \cite{Bo1986}.

In studies of $\mathcal{M}_{sph}$, what has been mainly concerned is the relation between the $L^p$ boundedness and geometric aspects of $\mathcal{M}_{sph}$.
Generally speaking, geometric properties of $\wh{d\sigma_{sph}}$ are one of many examples 
which allow maximal operators associated with the Fourier multipliers to be bounded on $L^p(\R^d)$.
Thus, one can naturally question that for which conditions on the Fourier multipliers $m$ 
$\mathcal{M}_m(f) := \sup_{t>0}\big| (m(t\cdot)\wh{f}(\cdot))\check{\,}\big|$ is bounded on $L^p$.

In \cite{Ru1986}, Rubio de Francia showed mapping properties of $\mathcal{M}_m$ when $m$ satisfies the limited decay condition, 
$i.e$, 
$$
  \text{$|\partial^\gamma m(\xi)| \leq C |\xi|^{-a}$ for a given $a>0$,}
$$
which is a generalisation of geometric properties of $\wh{d\sigma_{sph}}$.
Results in \cite{Ru1986} are based on the following square function estimates:
\begin{customthm}{A}[\cite{Ru1986}, Lemma 4]\label{thm LD}
  Let $m$ be a function of class $C^s(\R^d)$ with $s>\frac{d}{2}$ and supported in $\{ 1/2< |\xi| <2\}$,
  and let the $g$-function be given by 
  $$
  G_m(f)(x) = \biggl( \int_0^\infty \big| \big( m(t\xi) \wh{f}(\xi) \big)\check{\,}(x) \big|^2 \frac{dt}{t}\biggr)^{1/2}.
  $$
  Then for $\beta > d/2$ we have
  $$
  \big\| G_m (f) \|_{L^1(\R^d)} \leq  C_\beta \| m \|_{L_\beta^2(\R^d)} \|f\|_{H^1(\R^d)}.
  $$
\end{customthm}

We also introduce results of Dappa and Trebels \cite{Da_Tre1985}, which could be understood as a radial analogue of Theorem \ref{thm LD}.
\begin{customthm}{B}[\cite{Da_Tre1985},Theorem 1]\label{thm DaTr}
  Let $\rho \in C^{[d/2 +1]}(\R^d \setminus \{0\})$ be a positive, $A_t$-homogeneous distance function. 
  Let $m$ be a measurable function on $(0, \infty)$ which vanishes at infinity.
  \begin{enumerate}
    \item If $m$ satisfies 
    \begin{align}
      &\int_0^\infty t^{\lambda-1}|m^{(\lambda)}(t)|dt \leq B,\label{trebels1}\\
      &\bigg(\int_0^\infty |t^\lambda m^{(\lambda)}(t)|^2 \frac{dt}{t} \bigg)^{1/2} \leq B,\nonumber
    \end{align}
    for $\lambda > d\big| \frac{1}{p} - \frac{1}{2}\big|$, then $\mathcal{M}_{m\circ\rho}$ is of strong type $(p,p)$, $p\in (1,\infty)$, 
    also $\mathcal{M}_{m\circ\rho}$ is of weak type $(1,1)$.
    \item If $m$ satisfies \eqref{trebels1} for $\lambda>(d-1)\bigl(\frac{1}{2} - \frac{1}{p}\bigr) +1$, 
    then $\mathcal{M}_{m\circ\rho}$ is of strong type $(p,p)$, $2\leq p \leq \infty$.
  \end{enumerate}
\end{customthm}
By $A_t$-homogeneous distance, we mean a continuous function on $\R^d$ with 
$$
  \rho>0,\q \rho(A_t \xi) = t\rho(\xi),\q\text{for all $t>0$, $\xi\in\R^d$},
$$
where $A_t$ is a dilation matrix given by $A_t := e^{P \log{t}} = \sum_{n=0}^\infty \frac{(\log{t})^n}{n!}P^n$ for $P = (p_{ij})$ introduced in \cite{St_Wa1978}.
Since one of examples of $m\circ\rho$ is a radial Fourier multiplier, 
our main result (Theorem \ref{main}) can be understood as a non-radial generalisation for $m\circ\rho$.
When $m$ is radial, Theorem \ref{thm DaTr} is extended to spectral multipliers and we recommend \cite{De_Kr2022, Ma_Me1990} and references therein.

In \cite{Ch_Gr_Hon_See2005}, Christ, Grafakos, Hoz\'ik and Seeger studied $\mathcal{M}_m$ when $m$ is the H\"ormander-Mikhlin multiplier. 
\begin{customthm}{C}[\cite{Ch_Gr_Hon_See2005}, Corollary 1.3]\label{thm GS}
  Suppose $1<p<\infty, r = \min\{p,2\}$, and $\beta>d/r$. Suppose that 
  \begin{align}
    \|m(2^k \cdot) \wh{\psi}(\cdot) \|_{L_\beta^r(\R^d)} &\leq \omega(k), \q k\in\Z, \nonumber
    \\
    \omega^*(0) + \sum_{l=1}^\infty \frac{\omega^*(l)}{l} &< \infty. \label{omega sum}
  \end{align}
  Then $\mathcal{M}_m$ is bounded on $L^p(\R^d)$.
  Note that $\omega^*(t) := \sup\{\lambda>0 : \text{card}\big( \{ k :|\omega(k)| > \lambda \}\big) >t \}$.
  Moreover, when $p=\infty$ $\mathcal{M}_m$ maps $L^\infty$ to $BMO$.
\end{customthm}

\begin{remark}
  In \cite{Ch_Gr_Hon_See2005} the authors also present a counter example which yields that 
  being the Mikhlin multiplier is not sufficient for $m$ to guarantee the $L^p$ mapping property of $\mathcal{M}_m$.
  In \cite{Gr_Hon_See2006}, Grafakos, Hoz\'ik and Seeger proved that the same result holds 
  when the denominator $l$ in the summation of \eqref{omega sum} is replaced into $l\sqrt{\log(l)}$ 
  and $\alpha > d/p +1/p'$ for $p\in (1,2]$, $\alpha > d/2 + 1/p$ for $p>2$.
  For more details in this direction, we recommend \cite{Ch_Gr_Hon_See2005,Gr_Hon_See2006} and references therein. 
\end{remark}

Considering historical results, one of novelties of this paper is that we obtain a low regularity result in terms of Theorem \ref{thm GS}  with an operator norm derived from Theorems \ref{thm LD} and \ref{thm DaTr}.
We should note that our result, Theorem \ref{main}, doesn't yield an improvement for Theorem \ref{thm GS}.
However, since Theorem \ref{main} requires reduced regularity of $m$, so the regularity difference allows us to consider not only Mikhlin type multipliers but also certain singular Fourier multipliers introduced by Miyachi in \cite{Mi1980} and Fourier multipliers of limited decay in \cite{Ru1986}.
To do so, we first modify square function estimates for $\mathcal{M}_m$ by making use of fractional calculus and bilinear complex interpolation.
For the bilinear interpolation, we consider the square function as a bilinear operator for $m$ and $f$.
Also we introduce an appropriate function space for $m$ to apply the interpolation.

To state main theorem, we recall notations for maximal Fourier multiplier operators.
\begin{align*}
  \mathcal{M}_m(f)(x) &:= \sup_{t>0} \big| T_{m(t\cdot)}(f)(x) \big|,\\
  T_{m}(f)(x) &:= \int_{\R^d} e^{ 2\pi i x \cdot \xi} m(\xi) \wh{f}(\xi) d\xi.
\end{align*}
We also make use of a function space $\Sigma^2(\mathcal{B})$ whose norm is given by 
$$
  \|f \|_{\Sigma^2(\mathcal{B})} 
  := \biggl( \sum_{j\in \Z}  \big\| f(2^j\cdot) \wh{\psi}(\cdot) \big\|_{\mathcal{B}}^2 \biggr)^{1/2},
$$
where $\supp(\wh{\psi})$ is in $\{\frac{1}{2} < |\xi| <2 \}$ as defined in \eqref{psi} and $\mathcal{B}$ denotes a Banach space of functions on $\R^d$.
Properties of $\Sigma^2(\mathcal{B})$ are introduced in Section \ref{Preliminaries}.
In most cases, we choose $\mathcal{B} = L_s^p$ or $B_{p,p}^s$.
We note that $B_{p,q}^s(\R^d)$ is a space of tempered distributions $f$ whose norm is given by
$$
\|f\|_{B_{p,q}^s(\R^d)} := \| S_0 f\|_{L^p(\R^d)} + \bigg( \sum_{j =1}^\infty \big( 2^{js} \| \psi_j \ast f\|_{L^p(\R^d)} \big)^q \bigg)^{1/q},\q S_0 f := \sum_{j \leq 0} \psi_j \ast f,
$$
where $\psi_j(\cdot) = 2^{jd} \psi(2^j \cdot)$.
Then for $p=q$ we let $B_{p,p}^s(\R^d) = B_p^s(\R^d)$, and $B_p^s(\R^d)$ satisfies the embedding property $L_s^p(\R^d) \subseteq B_p^s(\R^d)$ for $p\geq2$ and $s\in \R$.
Now we introduce our main theorem. 
\begin{theorem}\label{main}
  Let  $p \in (1,\infty)$. 
  If $m\in \Sigma^2 (B_{p_0}^s)$, $\frac{1}{p_0} = |\frac{1}{p} - \frac{1}{2}|$ and $s>d\big|\frac{1}{p} - \frac{1}{2}\big|+\frac{1}{2}$,
  then we have
  \begin{align*}
    \big\| \mathcal{M}_m(f) \big\|_{L^p(\R^d)} \lesssim \|m\|_{\Sigma^2 (B_{p_0}^s)} \|f\|_{L^p(\R^d)}.
  \end{align*}
  Note that $B_{p_0}^s = B_{p_0, p_0}^s(\R^d)$ denotes the Besov space on $\R^d$. 
  Moreover, when $p=1$ or $p=\infty$, $\mathcal{M}_m$ satisfies $(H^1 \to L^1)$, $(L^\infty \to BMO)$ estimates, respectively.
\end{theorem}
In proving Theorem \ref{main}, we first obtain a pointwise upper bound of $\mathcal{M}_m$ by a Hilbert space valued operator $T_{\widetilde{m}}$, where $\widetilde{m}$ is a variant of $m$ by using the fractional calculus.
We note that the vector-valued operator satisfies $(L^1 \to H^1)$, $(BMO \to L^\infty)$ estimates due to \cite{Ru_Ru_To1986}.
Then we consider the vector-valued operator $T_{\widetilde{m}}(f)$ as a bilinear operator for $m$ and $f$, and make use of multilinear interpolation introduced by Calder\'on \cite[paragraph 10.1]{Cal1964}. 
We also introduce an appropriate function space for $m$ and $\widetilde{m}$. The function space and the multilinear perspective allows us to apply the multilinear interpolation, which yields interpolation among conditions on $m$ so that we obtain condition for $\mathcal{M}_m$ to be bounded on each $L^p(\R^d)$.

\begin{remark}\label{Remark 2}
  For $m$ in Theorem \ref{main}, we can say that $\|m\|_{\Sigma^2 (C(\R^d))}<\infty$ since $B_{p_0}^s(\R^d) \subset C(\R^d)$.
  Thus it follows that $m(\xi) \to 0$ as $|\xi| \to 0$ or $|\xi| \to \infty$.
  If one considers $m$ such that $m(\xi) \not\to 0$ as $|\xi| \to 0$, Theorem \ref{main} still holds if $m$ is of class $C^\infty$ near the origin.
  Hence we have 
  $$
  T^*(f) := \sup_{t>0} \big| \big( m(t\cdot)\varphi_0(t\cdot) \wh{f}(\cdot) \big)\check{\,} \big| \lesssim \mathcal{M}f,
  $$
  where $\varphi_0\in C^\infty(\R^d)$, $\varphi_0 \equiv 1$ in $B(0,1)$ and vanishes outside of $B(0,2)$. 
 Although $m$ is not smooth, if $m$ is of class $C^{[\frac{d}{2}] +2}(\R^d)$, then one can use \cite[Theorem II.2.1]{Ru_Ru_To1986} to obtain the $L^p$ boundedness of $T^*$ for $1<p\leq \infty$.

  The novelty of Theorem \ref{main} is that we actually obtain maximal estimates for $\cM_m$ when $m$ is not differentiable at the origin; the condition of Theorem \ref{main} allows us to consider $m$ whose derivatives blow up at the origin.
  For examples of these kinds, we present Propositions \ref{prop Lp cvgc} and \ref{prop: PWC}.
\end{remark}

  Even if $m$ is not smooth and is not of class $C^{[\frac{d}{2}] +2}(\R^d)$, $L^p$ estimates for $\mathcal{M}_m$ still follow from Theorem~\ref{main}.
\begin{corollary}\label{2205151153}
Let $p\in(1,\infty)$, and $m$ be a function of class $B^s_{p_0}(\mathbb{R}^d)$ where $s>d\big|\frac{1}{p}-\frac{1}{2}\big|+\frac{1}{2}$.
Then, for any fixed $\phi_0\in C_c^{\infty}(\mathbb{R}^d)$ we have
$$
\|\mathcal{M}_{m\phi_0}(f)\|_{L^p(\mathbb{R}^d)}\lesssim \|m\|_{B^s_{p_0}(\mathbb{R}^d)}\|f\|_{L_p(\mathbb{R}^d)}\,.
$$
Moreover, when $p=1$ or $p=\infty$, $\cM_{m\phi_0}$ satisfies $(H^1 \to L^1)$, $(L^\infty \to BMO)$ estimates, respectively.
In particular, if $m$ is a function of class $L^2_s$ with $s>\frac{d+1}{2}$ and has compact support, then $\|\mathcal{M}_m\|_{L^p\rightarrow L^p}$ is bounded for any $p\in(1,\infty)$. 
\end{corollary}

Due to $\frac{d+1}{2}< [\frac{d}{2}]+2$,  Corollary~\ref{2205151153} improves the condition in Remark~\ref{Remark 2} that $m$ is of class  $C^{[\frac{d}{2}] +2}(\R^d)$.
To our best knowledge, Corollary~\ref{2205151153} is the first improvement for $m$ supported near the origin since Rubio de Francia suggested the condition $m\in C^{[\frac{d}{2}] +2}(\R^d)$ in \cite{Ru1986}.

Now we introduce distinguishing features of Theorem \ref{main} in terms of Theorem \ref{thm DaTr} and Theorem \ref{thm GS}.
In \cite{Da_Tre1985}, the authors actually suggested a condition for non-radial $m$ 
  under which $\mathcal{M}_m$ is of strong type $(p,p)$ for $1\in (1, \infty)$.
  We note that Theorem \ref{main} requires $m$ to be regular in terms of $s> \frac{d+1}{2}$ for $L^p$ boundedness for all $p\in(1,\infty)$,
  while conditions in \cite{Da_Tre1985} require $s> d$ for the same result.
For results of radial $m$ in \cite{Da_Tre1985}, Theorem \ref{main} yields the following counter part:
\begin{corollary}\label{cor: radial}
  Let $m$ be a radial symmetric function given by $m(\xi) = h(|\xi|)$ and $s>d\big| \frac{1}{p} - \frac{1}{2}\big| + \frac{1}{2}$.
  Then we have
  \begin{align}\label{22.04.05.1}
    \|\mathcal{M}_m(f)\|_{L^p(\R^d)}\lesssim \|h\|_{{\Sigma^2(L_s^2(\R))}}\|f\|_{L^p(\R^d)}.
  \end{align}
\end{corollary}
Corollary \ref{cor: radial} is similar with the result of \cite{Da_Tre1985} in the radial symmetric case.
We note that \eqref{22.04.05.1} maybe useful in some situations due to the norm equivalence such as Lemma \ref{22.03.10.2} or Proposition \ref{prop:equiv}.

We now investigate results of Theorem \ref{main} and Theorem \ref{thm GS}, which is \cite[Corollary 1.3]{Ch_Gr_Hon_See2005}.
In \cite[Corollary 1.3]{Ch_Gr_Hon_See2005} for $p,q\in (1, \infty)$ the condition 
\begin{align}\label{condi:CGHS}
  \sum_{j\in\Z} \| m(2^j \cdot) \wh{\psi}(\cdot) \|_{L_s^r}^q <\infty,\q r=\min\{p,2\}, \,\, s> \frac{d}{r}
\end{align}
yields the $L^p$ boundedness of $\mathcal{M}_m$.
Recall that in Theorem \ref{main}, $\mathcal{M}_m$ is bounded on $L^p(\R^d)$ whenever
\begin{align}\label{condi:LeeSeo}
  \sum_{j\in\Z} \| m(2^j \cdot) \wh{\psi}(\cdot) \|_{B_{p_0}^s}^2   <\infty,\q \frac{1}{p_0} = \Big|\frac{1}{p} - \frac{1}{2} \Big|,\,\, s>d\Big|\frac{1}{p}-\frac{1}{2}\Big| +\frac{1}{2}.
\end{align}
It should be noted that the quantities \eqref{condi:CGHS} and  \eqref{condi:LeeSeo} cannot be directly compared.
When $p\leq 2$, however, comparing the ranges of $s$ in \eqref{condi:CGHS} and \eqref{condi:LeeSeo}, one can check a gain of regularity by $\frac{d-1}{2}$. 
Such gain of regularity allows us to consider non-Mikhlin type Fourier multipliers.  For examples of non-Mikhlin type, we suggest singular Fourier multipliers in \cite{Mi1980} or multipliers of limited decay in \cite{Ru1986}, which cannot be handled by the condition \eqref{condi:CGHS}.
We also present a convergence type result for those Fourier multipliers, which are Propositions \ref{prop Lp cvgc} and \ref{prop: PWC}.

On the other hand, for $p>2d$ \eqref{condi:LeeSeo} is dominated by \eqref{condi:CGHS} with $q=2$, $r=2$, $\alpha>d/2$, so one can actually improve Theorem \ref{main} by interpolating $L^2$ result of this paper and $L^\infty \to BMO$ result of  \cite{Ch_Gr_Hon_See2005}.
We should remark that the interpolation yields
\begin{align}\label{p geq 2}
  \|\mathcal{M}_m\|_{L^p \to L^p}\lesssim \| m\|_{\Sigma^2(B_{p_0}^s)},\q\text{for}\q s> d\Big(\frac{1}{2} - \frac{1}{p}\Big) + \frac{1}{p}, \q \text{and $p\geq2$},
\end{align}
where we adapt interpolation for vector-valued Hardy spaces and BMO in \cite{Bla_Xu1991}.
Note that \eqref{p geq 2} yields a gain of regularity by $\frac{1}{2} - \frac{1}{p}$ compared to \eqref{condi:LeeSeo} whenever $p\geq2$.
Due to \eqref{p geq 2}, we remark that every results derived from Theorem \ref{main} can be improved when $p\geq2$.

Now we introduce applications of Theorem \ref{main} whose proofs are given in Section \ref{Applications}.
First of all, we define the notion of Fourier multipliers with $slow$ $decay$. A function $m$ is said to be a Fourier multiplier with slow decay if $m$ satisfies
\begin{align}\label{slow decay}
  |\partial^\gamma m(\xi) | \leq C |\xi|^{- \delta|\gamma|}, \q \text{for some $\delta\in(0,1)$}.
\end{align}
It can be directly checked that the slow decay condition \eqref{slow decay} cannot be implied by the Mikhlin condition. 
We also give a non-trivial example of multipliers with slow decay. 
Let $m_\beta$ be a Mikhlin multiplier such that $|\partial^\gamma m_\beta(\xi)| \lesssim |\xi|^{-\beta - |\gamma|}$ and $m_\beta$ vanishes near the origin. 
Then for $\alpha \in(0,1)$ we define $m_{\alpha,\beta}$, $\mathfrak{M}_{\alpha,\beta}$ as following:
\begin{align}
  m_{\alpha,\beta}(\xi) &:= e^{i|\xi|^\alpha}m_\beta(\xi),\label{ineq:slow decay}\\
  \mathfrak{M}_{\alpha,\beta}(f)(x) &:= \sup_{t>0} \big| T_{m_{\alpha,\beta}(t\cdot)}(f)(x)\big|\nonumber.
\end{align}
Then $m_{\alpha, \beta}$ is slowly decaying with $\delta = 1-\alpha$ in \eqref{slow decay}.
In terms of Theorem \ref{thm GS}, 
one can see that $\mathfrak{M}_{\alpha, \beta}$ is bounded on $L^p(\R^d)$ for $\frac{1}{p}< \frac{\beta/\alpha}{d}$.
Applying Theorem \ref{main} to $\mathfrak{M}_{\alpha, \beta}$, however, we can make an improvement on the range of $p$. 
\begin{corollary}\label{cor half wave}
  $\mathfrak{M}_{\alpha, \beta}$ is bounded on $L^p(\R^d)$ for $|\frac{1}{p} - \frac{1}{2}|< \frac{\beta/\alpha}{d} - \frac{1}{2d}$.
\end{corollary}
The proof of Corollary \ref{cor half wave} is given in Section \ref{Applications}.
Here we present a breif reasoning for the range $\frac{1}{p}< \frac{\beta/\alpha}{d} + \frac{d-1}{2d}$.
It is straight forward to see that $|\partial^\gamma m_{\alpha,\beta}(\xi)| \lesssim |\xi|^{-(\beta + (1-\alpha)k)}$ for $|\gamma|=k$.
To apply Theorem \ref{main} we check $k = \beta + (1-\alpha)k$, so $k = \beta/\alpha$.
Then we must have $k = \beta/\alpha > d(\frac{1}{p} - \frac{1}{2}) + \frac{1}{2}$, 
which gives $\frac{1}{p} < \frac{\beta/\alpha}{d} + \frac{d-1}{2d}$.
In \cite{Mi1980}, one can find that 
for radial $m_{\alpha, \beta}$ the maximal operator $\mathfrak{M}_{\alpha, \beta}$ is bounded on $L^p(\R^d)$ for $p \in (1, \infty)$ 
when $\beta > \frac{d\alpha}{2}$ since the Fourier transform of $m_{\alpha, \beta}$ is in $L^1$ for those $\beta$, $\alpha$.
In our case, for $m_{\alpha, \beta}$ which is not necessarilly radial $\mathfrak{M}_{\alpha, \beta}$ is bounded on $L^p(\R^d)$ for $p \in (1, \infty)$ 
when $\beta > \frac{d\alpha}{2} + \frac{\alpha}{2}$.
One can understand the difference $\frac{\alpha}{2}$ is a radial counter part.

As an application of Corollary \ref{cor half wave}, we can study $e^{-it(-\Delta)^{\frac{\alpha}{2}}}$, which is a solution of a half wave equation
or fractional Schr\"odinger equations. 
Particularly, for $\alpha=1$ one can relate $e^{-it(-\Delta)^{\frac{1}{2}}}$ to a Fourier intergral operator of degree 1,
which is a central object in harmonic analysis.
In studies on dispersive equations, the Fourier integral operators, and the related topics, 
$e^{-it(-\Delta)^{\frac{\alpha}{2}}}(f)$ is considered as an extension operator $\mathcal{F}f(x,t)$, 
and $\big( L^2 \to L_t^q L_x^r\big)$ estimates of $\mathcal{F}f$ is one of main goals of studying Fourier integral operators and dispersive equations.
In this paper, we do not discuss the mixed norm estimates for $e^{-it(-\Delta)^{\frac{\alpha}{2}}}$, and for $\alpha=1$ we recommend \cite{Hor1971,Mock_See_So1993} and references therein.
Instead of the mixed norm estimates, we are interested in $L^p$ behavior of 
$$
U_{\alpha,\beta}(f)(x,t) = \frac{e^{-it(-\Delta)^{\frac{\alpha}{2}}} - I}{t^\beta}(f)(x),\q\text{as $t\to0$.}
$$
\begin{proposition}\label{prop Lp cvgc}
  For $\alpha \in (0,1)$ and $\beta\in(1/2,1)$, $U_{\alpha, \beta}(f)(x,t) \to 0$ in $L^p(\R^d)$ as $t\to 0$,
  whenever $\frac{d-2\beta+1}{2d}< \frac{1}{p}<\frac{d+2\beta-1}{2d}$, and $f\in \dot{L}_{\alpha\beta}^p(\R^d)$.
  In particular, under the same condition on $\alpha, \beta$ and $f$, we have for $a.e.$ $x$
  $$
  	\Big| e^{-it(-\Delta)^{\frac{\alpha}{2}}}f(x) - f(x) \Big| = O(t^\beta),\q t\to0 .
  $$
  Note that $f \in \dot{L}_s^p$ is equivalent to $(|\cdot|^s \wh{f})\check{\,} \in L^p$.
\end{proposition}

Theorem \ref{main} can be also applied to $\mathcal{M}_m$ when $m$ enjoys both Mikhlin's condition and the limited decay. 
Let $a>0$, $b >\frac{1}{2}$. We assume that $m_{a,b}$ is of class $L^2_{s, loc}$ for $s>\frac{d+1}{2}$ and satisfies that for given $p\in(1,\infty)$ and $j>0$,
\begin{align}\label{condi: M LD}
  \| m_{a,b}(2^j\cdot) \wh{\psi}(\cdot) \|_{L_s^{p_0}(\R^d)} \lesssim 2^{-j \min(a, b-s)},\q 
  \frac{1}{p_0} = \big|\frac{1}{p} - \frac{1}{2}\big|,\q 
  s>d\big| \frac{1}{p} - \frac{1}{2}\big| +\frac{1}{2}.
\end{align}
Under the condition \eqref{condi: M LD} we have
\begin{corollary}\label{cor: M LD}
	Let $m_{a,b}$ be a Fourier multiplier satisfying \eqref{condi: M LD}. 
  Then $\mathcal{M}_{a,b}$ is bounded on $L^p(\R^d)$ for $\frac{d+1 -2b}{2d} < \frac{1}{p} < \frac{d-1 +2b}{2d}$.
\end{corollary}
The condition \eqref{condi: M LD} can be understood in a two-fold manner.
If $m$ satisfies 
\begin{align}\label{condi: M}
\| m(2^j\cdot) \wh{\psi}(\cdot) \|_{L_s^{p_0}(\R^d)} \lesssim 2^{-ja},\q j>0,
\end{align} 
then $m$ can be regarded as a Mikhlin multiplier such that $|\partial^\gamma m(\xi)| \lesssim (1+|\xi|)^{-a} |\xi|^{-|\gamma|}$ and it satisfies conditions of Theorem \ref{thm GS}, which guarantees that $\mathcal{M}_m$ to be bounded on $L^p$ for certain $p$.
On the other hand, if we have
\begin{align}\label{condi: LD}
\| m(2^j\cdot) \wh{\psi}(\cdot) \|_{L_s^{p_0}(\R^d)} \lesssim 2^{-j(b-s)},\q j>0,
\end{align}
then $m$ is a multiplier of the limited decay whose maximal estimates are introduced in \cite{Ru1986}.
It can be checked that \eqref{condi: M} behaves worse than \eqref{condi: LD} for $s< b-a$ and \eqref{condi: LD} enjoys worse decay than that of \eqref{condi: M} for $s> b-a$.
Thus, \eqref{condi: M LD} considers the worst case when the cases \eqref{condi: M} and \eqref{condi: LD} are combined.

Using Corollary \ref{cor: M LD}, we obtain a pointwise convergence type result.
For $m \in C_{loc}^{[\frac{d+1}{2}]+1}(\R^d)$ such that $|\partial^\gamma m(\xi)| \leq (1+|\xi|)^{-\beta}$ for any multi-index $\gamma$, and $m(0) = 1$, we have
\begin{proposition}\label{prop: PWC}
For $\alpha \in (0,1)$ and $f \in \dot{L}_{\alpha}^p(\R^d)$ with $\frac{d+1 -2(\alpha+\beta)}{2d} < \frac{1}{p} < \frac{d-1 +2(\alpha + \beta)}{2d}$,
we have
$$
	\lim_{t\to 0} \frac{f(x) - T_{m(t\cdot)}f(x)}{t^\alpha} = 0,\q\text{for almost every $x \in\R^d$}.
$$
\end{proposition}
A noteworthy example for $m$ in Proposition \ref{prop: PWC} is the Fourier transform of a surface-carried measure $\sigma$ on $S \subset\R^d$,
where $S$ is a hypersurface of nonvanishing $d-1$ principal curvatures.

\begin{notation}
	Let $A \lesssim B$ denote $A\leq CB$, where $C$ is independent of $A$ and $B$.
	$B(x,R)$ denotes a ball in $\R^d$ of radius $R$ centered at $x$.
\end{notation}

\section{Preliminaries}\label{Preliminaries}

\subsection{Function spaces}

Choose $\phi \in \mathscr{S}(\R^d)$ such that $\wh{\phi} \equiv 1$ on $B(0,1)$ and $\wh{\phi} \equiv 0$ on $B(0,2)^c$.
Define 
\begin{align}\label{psi}
\wh{\psi}_j = \wh{\psi}(\cdot/2^j) =\wh{\phi}(\cdot/2^j) - \wh{\phi}(\cdot/2^{j-1})\q \text{for}\q j\in\Z.
\end{align}
Then we introduce function spaces appeared in our arguments. 
Let $\mathcal{B}$ be a Banach space continuously embedded in the space of distributions $\mathscr{D}'(\R^d)$.
For $q\in(1,\infty)$ and $\theta\in\R$, we define normed spaces
\begin{align*}
  \Sigma_\theta^q(\mathcal{B})
  &:= \{f \in \mathscr{D}'(\R^d\setminus\{0\}) 
    : \|f \|_{\Sigma_\theta^q(\mathcal{B})}^q 
      := \sum_{j\in\Z} 2^{q\theta}\big\| f(2^j \cdot) \wh{\psi}(\cdot) \big\|_{\mathcal{B}}^q <\infty \},\\
    l_\theta^q(\mathcal{B})
    &:= \{ \{f_j\}_{j\in \mathbb{Z}} : f_j \in \mathcal{B},\,\, \text{and}\,\, \sum_{j\in\mathbb{Z}} 2^{q\theta} \|f_j\|_{\mathcal{B}} < \infty \}.
\end{align*}
Note that $\Sigma_\theta^q(\mathcal{B})$ is a Banach space where we additionally assume that
\begin{align}\label{ineq:22 03 21 12 58}
\|\wh{\psi} f\|_{\mathcal{B}}\lesssim \|f\|_{\mathcal{B}}\quad\text{and}\quad \|f(2\,\cdot\,)\|_{\mathcal{B}}\simeq  \|f\|_{\mathcal{B}}.
\end{align}

Let $S$ and $R$ be mappings given by 
\begin{align*}
  S(f) &= \{f(2^j \cdot) \widehat{\psi}(\cdot)\}_{j\in\Z} =: \{f_j\}_{j\in\Z},\\
  R\bigl(\{f_j\}\bigr) &= \sum_{j} \big(\widehat{\psi}_{j-1}+\widehat{\psi}_j+\widehat{\psi}_{j+1}\big)(\cdot) f_j(2^{-j}\cdot).
\end{align*}
Then by \eqref{ineq:22 03 21 12 58} $S  \in L\big(\Sigma_\theta^q(\mathcal{B}), l_\theta^q(\mathcal{B})\big)$ and $R \in L\big(l_\theta^q(\mathcal{B}),\Sigma_\theta^q(\mathcal{B})\big)$, respectively, where $L(A,B)$ denotes a space of bounded linear maps from $A$ to $B$.
Moreover, one can check that $R\circ S$ is an identity map of $Id_{\Sigma_\theta^q(\mathcal{B})}$.
Hence by \cite[Theorem 1.2.4]{Tri1995} $\Sigma_\theta^q(\mathcal{B})$ is isomorphic to a closed subspace of $l_\theta^q(\mathcal{B})$ which is a reflexive Banach space when $\mathcal{B}$ is reflexive. 
If $\theta=0$, then we simply write $\Sigma^q(\mathcal{B})$.

Let $(\mathcal{A}, \mathcal{B})$ be an interpolation couple of Banach spaces, $q = (1-\delta)q_0 + \delta q_1$, and $\theta=(1-\delta)\theta_0+\delta\theta_1$.
Then we have from \cite[Theorem 1.2.4 and 1.18.1]{Tri1995}
\begin{align}\label{intpltn:Sigma}
[\Sigma_{\theta_0}^{q_0}(\mathcal{A}), \Sigma_{\theta_1}^{q_1}(\mathcal{B})]_\delta = \Sigma_\theta^q\bigl([ \mathcal{A}, \mathcal{B} ]_\delta\bigr).
\end{align}
For instance, if we choose $\mathcal{A} = L_{s_1}^2$, $\mathcal{B} = L_{s_2}^2$, $s_1,s_2\geq0$ with $\theta=0$, $q=2$, then we have by \eqref{intpltn:Sigma}
\begin{align}\label{intpltn:Sigma Sobolev}
  [\Sigma^2(L_{s_1}^2), \Sigma^2(L_{s_2}^2)]_\delta = \Sigma^2(L_s^2),\q s = (1-\delta)s_1 + \delta s_2,\, \delta\in(0,1).
\end{align}

Another example for $\Sigma^q(\mathcal{B})$ is the weighted Sobolev space $H_\theta^{p,\gamma}$ introduced in \cite{Lot_2000}, which equals $\Sigma^p_{\theta/p}(L^p_{\gamma})$.
When $\gamma\in\mathbb{N}\cup\{0\}$ we have
\begin{align}\label{w sobolev}
  H_\theta^{p, \gamma}(\R^d \setminus \{0\}) 
  = \{f \in \mathscr{S}'(\R^d) : \|f \|_{H_\theta^{p,\gamma}(\R^d\setminus \{0\})}^p := \sum_{l=0}^{\gamma}  \int_{\R^d} \big| D_x^l (f)(x) \big|^p |x|^{pl+\theta-d}  dx<\infty \}.
\end{align}
Indeed, one can check that by the change of variables,
\begin{align*}
\sum_{l=0}^{\gamma}   \int_{\R^d} \big| D_x^l (f)(x) \big|^p |x|^{pl+\theta-d}  dx  
&\simeq \sum_{l=0}^{\gamma} \sum_{j\in\Z}   2^{j(pl + \theta -d)} \int_{\R^d} \big| D_x^l (f)(x) \big|^p |\widehat{\psi}_j(x)|^p  dx\\
 &= \sum_{l=0}^{\gamma}  \sum_{j\in\Z}   2^{j\theta} \int_{\R^d} \big| D_x^l(f(2^j\cdot)) \big|^p |\widehat{\psi}(x)|^p dx\\
  &\simeq\sum_{l=0}^{\gamma}  \sum_{j\in\Z}  2^{j\theta} \int_{\R^d} \big| D_x^l( f(2^j\cdot)\widehat{\psi}(\cdot) )(x) \big|^p dx.
\end{align*}
Note that $|\cdot|^{-d}$ is just a scaling factor and for $\theta=0$ it follows that $H_0^{p,\gamma}$-norm is scaling invariant.
For more information of $H_\theta^{p, \gamma}$, we recommend \cite{Lot_2000} and references therein.

\subsection{Fractional calculus}

To study a maximal operator, classical square function argument makes use of the fundamental theorem of calculus with respect to $t>0$.
In doing so, one must take derivative with respect to $t$, which in turn yields loss of regularity of a symbol.
Such loss may be optimized if we use a fractional analogue of the fundamental theorem.

We define the Riemann-Liouville integrals and fractional derivatives of order $\alpha \in (0,1)$. 
For all $t \in (0,\infty)$
\begin{align*}
  I_{0+}^\alpha f(t) &:= \frac{1}{\Gamma(\alpha)} \int_0^t (t-s)^{\alpha-1} f(s) ds,\\
  D_{0+}^{\alpha}F(t) &:= \frac{d}{dt}\big(I_{0+}^{1-\alpha}F\big)(t)
\end{align*}
where $f$ is a locally integrable function on $[0,\infty)$ and $F$ satisfies $I_{0+}^{1-\alpha}F$ is absolutely continuous.
For further use of $I_{0+}^\alpha$ and $D_{0+}^\alpha$, we introduce the following lemma which is crucial in the proof of Theorem \ref{main}:
\begin{lemma}\label{220514340}
Let $F\in C_{loc}\big([0,\infty)\big)\cap C^{0,\beta}_{loc}\big((0,\infty)\big)$ for a fixed $\beta\in(0,1]$.
Then for any $\alpha\in(0,\beta)$, $D^{\alpha}_{0+}F$ satisfies the following identity:
\begin{align}\label{220514332}
D^{\alpha}_{0+}F(t)=\frac{1}{\Gamma(1-\alpha)}\left(\frac{F(t)}{t^{\alpha}}+\int_0^t\frac{F(t)-F(s)}{(t-s)^{1+\alpha}}ds\right)\,.
\end{align}
Moreover, we have
\begin{align}\label{220514330}
F(t)=I_{0+}^{\alpha}D_{0+}^{\alpha}F(t)+\frac{t^{\alpha}}{\Gamma(\alpha)}F(0)\,.
\end{align}
\end{lemma}
\begin{proof}
From the definition of $I^{1-\alpha}F(t)$, for $h>0$ we have
\begin{align*}
&\frac{I^{1-\alpha} F(t+h)-I^{1-\alpha}F(t)}{h}\\
=\,&\frac{1}{h}\int_t^{t+h}\big((t+h-s)^{-\alpha})\big(F(s)-F(t)\big)ds\\
&+ \frac{1}{h}\int_0^t\big((t+h-s)^{-\alpha}-(t-s)^{-\alpha})\big(F(s)-F(t)\big)ds\\
&+\frac{1}{h}F(t)\Big(\int_0^{t+h}(t+h-s)^{-\alpha}-\int_0^t(t-s)^{-\alpha}ds\Big)\\
=:&I_1+I_2+I_3\,.
\end{align*}
Define
$$
A=\sup_{0\leq s\leq 2t}|F(s)|\quad\text{and}\quad B=\sup_{t/2\leq s\leq 2t}\frac{|F(t)-F(s)|}{|t-s|^{\beta}}.
$$
For $0<h< t/2$, by change of variables, we have
\begin{align*}
|I_1|\leq \int_0^1(hr)^{-\alpha}\big|F(t+h(1-r))-F(t)\big|ds\leq h^{\beta-\alpha} B \int_0^1r^{-\alpha}(1-r)^{\beta}dr.
\end{align*}
Therefore, $I_1\rightarrow 0$ as $h\rightarrow 0$.
For $I_2$, we make use of the fundamental theorem of calculus to obtain
\begin{align*}
I_2=-\alpha\int_0^t\int_0^1(t-s+hr)^{-1-\alpha}dr\big(F(s)-F(t)\big)ds\,.
\end{align*}
Note that
\begin{align*}
&\int_0^t\int_0^1(t-s+hr)^{-1-\alpha}\big|F(s)-F(t)\big|drds\\
\leq\,&\int_0^t(t-s)^{-1-\alpha}|F(s)-F(t)|ds\\
\lesssim \,& A t^{-\alpha}+ B t^{\beta-\alpha}<\infty\,.
\end{align*}
Thus by the dominated convergence theorem, it follows that 
$$
I_2 \rightarrow -\alpha\int_0^t(t-s)^{-1-\alpha}\big(F(s)-F(t)\big)ds\q\text{as}\q h\to 0.
$$
For $I_3$, by direct calculation we obtain
$$
I_3=\frac{F(t)}{1-\alpha}\times \frac{(t+h)^{1-\alpha}-t^{1-\alpha}}{h}\rightarrow \frac{F(t)}{t^{\alpha}}\,.
$$
For $\frac{I^{1-\alpha} F(t)-I^{1-\alpha}F(t-h)}{h}$ and $h>0$, we have
\begin{align*}
\frac{I^{1-\alpha} F(t)-I^{1-\alpha}F(t-h)}{h}
=\,&\frac{1}{h}\int_{t-h}^t \big((t-s)^{-\alpha})\big(F(s)-F(t-h)\big)ds\\
&+ \frac{1}{h}\int_0^{t-h}\big((t-s)^{-\alpha}-(t-h-s)^{-\alpha})\big(F(s)-F(t-h)\big)ds\\
&+\frac{1}{h}F(t-h)\Big(\int_0^{t}(t-s)^{-\alpha}-\int_0^{t-h}(t-h-s)^{-\alpha}ds\Big).
\end{align*}
Then by the same argument it follows that for $h>0$
$$
\frac{I^{1-\alpha} F(t)-I^{1-\alpha}F(t-h)}{h} \rightarrow \frac{F(t)}{t^{\alpha}}+\int_0^t\frac{F(t)-F(s)}{(t-s)^{1+\alpha}}ds \q\text{as}\q h\to 0\,.
$$
Therefore \eqref{220514332} is proved.

From the assumption of $F$ we obtain $D^{\alpha}_{0+}F(t)\in L^1_{loc}([0,\infty))$, and the equality \eqref{220514330} holds by \cite[Theorem 2.4]{Sam_Kil_Ma1993}.
The lemma is proved.
\end{proof}

\section{Proof of Theorem \ref{main}}

Let $\varepsilon\in(0,\frac{1}{6})$ be sufficiently small and  $s> d|\frac{1}{p}-\frac{1}{2}|+\frac{1}{2}+3\varepsilon$. Since $\Sigma^2(B^{s}_{p_0})\subset \Sigma^2(C^{0,\frac{1}{2}+3\varepsilon})$, we consider $m$ as an element of $C_{loc}^{0,\frac{1}{2}+3\varepsilon}(\mathbb{R}^d\setminus\{0\})$ and 
$$
\|m(2^j\cdot)\widehat{\psi}\|_{C^{0,1/2+3\varepsilon}}\rightarrow 0 \q\text{as}\q j\rightarrow -\infty.
$$
This implies that $m(\xi)\rightarrow 0$ as $\xi\rightarrow 0$. 
Thus we want to  show that for such $m$,
\begin{align*}
  \big\| \mathcal{M}_{m}(f) \big\|_{L^p(\R^d)} \lesssim \| m\|_{\Sigma^2(B_{p_0}^s)} \|f \|_{L^p(\R^d)}.
\end{align*}	
To do this we introduce a square function type estimate using fractional calculus.
From Lemma \ref{220514340} we obtain that
$$
m(t\xi)= \frac{1}{\Gamma(\frac{1}{2} - \varepsilon)}I^{\frac{1}{2}+\varepsilon}\big( (\cdot)^{-\frac{1}{2}-\varepsilon}\widetilde{m}(\cdot\xi)\big)(t),
$$
where
$$
\widetilde{m}(t\xi) := m(t\xi)+  (1/2+\varepsilon)\int_0^1 \frac{m(t\xi) - m(st\xi)}{(1-s)^{\frac{3}{2}+\varepsilon}} ds \,.
$$
Note that for fixed $\xi \in \R^d$ with $|\xi| \in (2^{j-1}, 2^{j+1})$ for some $j\in\Z$, we have
\begin{align*}
	| \widetilde{m}(\xi) | 
	\lesssim \| m \|_{L^\infty} + \Big| \int_0^1 \frac{m(\xi) - m(s\xi)}{(1-s)^{\frac{3}{2} + \varepsilon}} ds \Big|
	\lesssim \| m\|_{L^\infty} + \Big| \int_{1/2}^1 \frac{m(\xi) - m(s\xi)}{(1-s)^{\frac{3}{2} + \varepsilon}} ds \Big|.
\end{align*} 
Note that $\|m\|_{\infty}\lesssim \|m\|_{\Sigma(L^{\infty})}\lesssim\|m\|_{\Sigma(C^{0,1/2+3\varepsilon})}$.
We also have
\begin{align*}
 \Big| \int_{1/2}^1 \frac{m(\xi) - m(s\xi)}{(1-s)^{\frac{3}{2} + \varepsilon}} ds \Big| 
 &\lesssim |\xi|^{\frac{1}{2} + 3\varepsilon} \sup_{1/2 < s<1}\frac{|m(\xi) - m(s\xi)|}{|\xi - s\xi|^{\frac{1}{2} + 3\varepsilon}}\\
 &\lesssim 2^{j(\frac{1}{2} + 3\varepsilon)} \sum_{k=j-1}^{j+1} \big[ m \wh{\psi}_k \big]_{C^{0,1/2 + 3\varepsilon}}\\
 &\lesssim  \sum_{k=j-1}^{j+1}\| m(2^k\cdot) \wh{\psi}(\cdot) \|_{C^{0,1/2+ 3\varepsilon}} \lesssim \|m\|_{\Sigma^2(C^{0,1/2+3\varepsilon})} <\infty.
\end{align*}
Since we choose arbitrary $\xi \in \R^d$,  $\widetilde{m}$ is bounded.
By the boundedness of $\widetilde{m}$, we make use of Fubini's theorem to obtain
$$
T_{m(t\cdot)} f = \frac{1}{\Gamma(\frac{1}{2} - \varepsilon)} I^{\frac{1}{2}+\varepsilon}\big(t^{-\frac{1}{2}-\varepsilon}T_{\widetilde{m}(t\cdot)}f \big).
$$
Thus it follows that
\begin{equation}\label{22 05 14 19 26}
\begin{aligned}
  \big| T_{m(t\cdot)} f \big|^2 &= \big|  \frac{1}{\Gamma(\frac{1}{2} - \varepsilon)} I^{\frac{1}{2}+\varepsilon}\big( t^{-\frac{1}{2}-\varepsilon}T_{\widetilde{m}(t\cdot)}f \big) \big|^2\\
  &=\bigg| \frac{1}{\Gamma(\frac{1}{2}-\varepsilon)} \int_0^t (t-s)^{-\frac{1}{2}+\varepsilon} s^{-\frac{1}{2}-\varepsilon}T_{\widetilde{m}(s\cdot)}f  ds \bigg|^2\\
  &\lesssim \int_0^t (t-s)^{-1+2\varepsilon} s^{-2\varepsilon} ds 
      \times \int_0^t  \big| T_{\widetilde{m}(s\cdot)}f \big|^2 \frac{ds}{s}.
\end{aligned}
\end{equation}
By taking supremum over $t>0$ on \eqref{22 05 14 19 26}, we obtain
\begin{align}\label{2022 0114 sqftn}
  \big| \mathcal{M}_{m}f|^2 
  \lesssim \int_0^\infty  \big| T_{\widetilde{m}(t\cdot)}f\big|^2 \frac{dt}{t}
  =: G_{\widetilde{m}}(f)^2.
\end{align}
Thanks to \eqref{2022 0114 sqftn}, we obtain $L^2(\frac{dt}{t})$-valued operator instead of supremum over $t>0$ in price of $\widetilde{m}$. 
Since we desire a condition of $m$ for the $L^p$ boundedness, the following mapping property of $m \mapsto \widetilde{m}$ is required:

\begin{lemma}\label{key embedding}
  For $\beta \geq 0$
  \begin{align}
  \sup_{\xi\neq 0}\Big(\int_0^{\infty}|\widetilde{m}(t\xi)|^2\frac{dt}{t}\,\Big)^{1/2} &\lesssim \| m \|_{\Sigma^2(C^{0,1/2 + 3\varepsilon})},\label{ineq:22 04 04 13 17}\\
  \| \widetilde{m} \|_{\Sigma^2(L_\beta^2(\R^d))} &\lesssim \| m \|_{\Sigma^2(L_{\beta + 1/2 + 3\varepsilon}^2(\R^d))}.\label{ineq:22 04 04 13 18}
  \end{align}
\end{lemma}

Additionally, to perform bilinear interpolation, we need initial estimates for $G_m$ in terms of the left hand sides of \eqref{ineq:22 04 04 13 17} and \eqref{ineq:22 04 04 13 18}, which is Lemma \ref{modified sqftn}.
\begin{lemma}\label{modified sqftn}
For the square function	
  \begin{align*}
  G_m(f) = \bigg( \int_0^\infty \big|T_{m(t\cdot)}f \big|^2 \frac{dt}{t} \bigg)^{1/2},
  \end{align*}
  we have
  \begin{align*}
    &\big\| G_m(f) \big\|_{L^2(\R^d)} \lesssim \sup_{\xi\neq 0}\Big(\int_0^{\infty}|m(t\xi)|^2\frac{dt}{t}\,\Big)^{1/2} \|f\|_{L^2(\R^d)},\\
    &\big\| G_m(f) \big\|_{L^1(\R^d)} \lesssim \| m\|_{\Sigma^2(L_s^2)} \|f\|_{H^1(\R^d)},\\
    &\big\| G_m(f) \big\|_{BMO(\R^d)} \lesssim \| m\|_{\Sigma^2(L_s^2)} \|f\|_{L^\infty(\R^d)}.
  \end{align*}
\end{lemma}
For further reasoning, we temporarily assume Lemmas \ref{key embedding} and \ref{modified sqftn}.
We present their proofs in Sections \ref{sec; key embedding} and \ref{sec; mod sqftn}. 
Then by Lemmas \ref{key embedding}, \ref{modified sqftn}, 
we conclude that 
\begin{align}
  \big\| G_{\widetilde{m}}(f) \big\|_{L^2(\R^d)} &\lesssim \| m\|_{\Sigma^2(C^{0,1/2+3\varepsilon})} \|f\|_{L^2(\R^d)},\label{ineq:22.02.22.17.31}\\
    \big\| G_{\widetilde{m}}(f) \big\|_{L^1(\R^d)} &\lesssim \| m\|_{\Sigma^2(L_{(d+1)/2+3\varepsilon}^2)} \|f\|_{H^1(\R^d)}.\label{ineq:22.02.22.17.32}
\end{align}
We interpolate \eqref{ineq:22.02.22.17.31}, \eqref{ineq:22.02.22.17.32} in terms of Lemma \ref{complex interpolation} to obtain
\begin{align*}
  \big\| G_{\widetilde{m}}(f) \big\|_{L^p(\R^d)} \lesssim \| m\|_{\Sigma^2(B_{p_0}^s)} \|f\|_{L^p(\R^d)},
\end{align*}
for $p_0 = |1/p - 1/2|^{-1}$ and $s\geq d|\frac{1}{p} - \frac{1}{2}| + \frac{1}{2}+3\varepsilon$.
We present the details for the interpolation in subsection \ref{subsec intpltn} and note that the interpolation is valid since we consider $G_{\widetilde{m}}$ as an $L^2(\frac{dt}{t})$-valued  bilinear operator $B$ in the sense that
$$
B(m,f)(x,t):=T_{\widetilde{m}(t\,\cdot\,)}f(x)\quad\text{and}\quad \big\| B(m, f) \big\|_{L^2(\frac{dt}{t})} = G_{\widetilde{m}}(f).
$$
Then by \eqref{2022 0114 sqftn}, we have
\begin{align*}
  \big\| \mathcal{M}_m(f) \big\|_{L^p(\R^d)} \lesssim \| m\|_{\Sigma^2(B_{p_0}^s)} \|f\|_{L^p(\R^d)}.
\end{align*}
This proves the theorem.

\section{Proof of Lemma \ref{key embedding}}\label{sec; key embedding}

\subsection{Proof of \eqref{ineq:22 04 04 13 17}}
We first show that $\sup_{\xi\neq 0}\big(\int_0^{\infty}|\widetilde{m}(t\xi)|^2 \frac{dt}{t}\big)^{1/2}\lesssim \| m \|_{\Sigma^2(C^{0,1/2+ 3\varepsilon})}$.
Recall that 
$$
  \widetilde{m}(\xi)
    = m(\xi) + (1/2+\varepsilon)\int_0^1 (1-s)^{-\frac{3}{2}-\varepsilon} \bigl(m(\xi) - m(s\xi)\bigr) ds,
$$
where $m$ vanishes near the origin.
Then for $m$ we have
\begin{align*}
\sup_{\xi \not= 0} \int_0^\infty |m(t\xi)|^2 \frac{dt}{t}\,&\leq  \sum_{j\in\mathbb{Z}}\sup_{\xi \not= 0}\int_{2^j/|\xi|}^{2^{j+1}/|\xi|} |m(t\xi)|^2 \frac{dt}{t} \\
&\lesssim \sum_{j\in\mathbb{Z}}\|m(2^j\,\cdot\,)\widehat{\psi}\|_{\infty}^2\leq \sum_{j\in\mathbb{Z}}\|m(2^j\,\cdot\,)\widehat{\psi}\|_{C^{0,1/2+3\varepsilon}(\R^d)}^2,
\end{align*}
Thus, without loss of generality we set 
\begin{align}\label{22 05 11 14 38}
\widetilde{m}(\xi)
    = \int_0^1 (1-s)^{-\frac{3}{2}-\varepsilon} \bigl(m(\xi) - m(s\xi)\bigr) ds.
\end{align}
Due to
\begin{align*}
|\widetilde{m}(\xi)|\lesssim \int_0^{1/2}|m(\xi)|+|m(s\xi)|ds+\int_{1/2}^1(1-s)^{-\frac{3}{2}-\varepsilon}|m(\xi)-m(s\xi)|ds,
\end{align*}
we have
\begin{align}\label{22 04 03 1}
\int_0^{\infty}|\widetilde{m}(t\xi)|^2\frac{dt}{t}\lesssim \int_{0}^{\infty} \Big(|m(t\xi)|^2+I(t\xi)^2 \Big) \frac{dt}{t}= \sum_{j\in\Z}\int_{|\xi|^{-1}}^{2|\xi|^{-1}} \Big(|m(2^jt\xi)|^2+I(2^jt\xi)^2 \Big) \frac{dt}{t}
\end{align}
where
$$
I(\xi)=\int_{1/2}^1(1-s)^{-\frac{3}{2}-\varepsilon}|m(\xi)-m(s\xi)|ds.
$$

Put $\widehat{\eta}=\widehat{\psi}_{-1}+\widehat{\psi}_{0}+\widehat{\psi}_{1}$ which equals to 1 on $\{\xi\,:\,1/2\leq |\xi|\leq 2\}$.
For $|\xi|^{-1}\leq  t\leq  2|\xi|^{-1}$ and $1/2\leq s\leq 1$, we obtain
$$
\frac{1}{2}\leq |st\xi|\leq |t\xi|\leq 2.
$$
Therefore we have
\begin{align*}
|m(2^jt\xi)|\leq \|m(2^j\cdot)\widehat{\eta}\|_{\infty}\,\,,\,\,\,\,\frac{|m(2^jt\xi)-m(2^jst\xi)|}{\big((1-s)t|\xi|\big)^{1/2+3\varepsilon}}\leq \big[m(2^j\cdot)\widehat{\eta}\,\big]_{C^{0,1/2+3\varepsilon}},
\end{align*}
which implies
\begin{align}\label{22 04 03 3}
|m(2^jt\xi)|+I(2^jt\xi)\,&= |m(2^jt\xi)|+(t|\xi|)^{1/2+3\varepsilon}\int_{1/2}^1(1-s)^{-1+2\varepsilon}\frac{|m(2^jt\xi)-m(2^jst\xi)|}{\big((1-s)t|\xi|\big)^{1/2+3\varepsilon}}ds\nonumber\\
&\lesssim\|m(2^j\cdot)\widehat{\eta}\|_{C^{0,1/2+3\varepsilon}}\,.
\end{align}
By \eqref{22 04 03 1}, \eqref{22 04 03 3}, and that $C^{0,\frac{1}{2}+3\varepsilon}$ satisfying the property \eqref{ineq:22 03 21 12 58}, we have
$$
\int_0^{\infty}|\widetilde{m}(t\xi)|^2\frac{dt}{t}\lesssim \sum_j\|m(2^j\cdot)\widehat{\eta}\|^2_{C^{0,1/2+3\varepsilon}}\simeq \sum_j\|m(2^j\cdot)\widehat{\psi}\|_{C^{0,1/2+3\varepsilon}}^2.
$$

%

\subsection{Proof of \eqref{ineq:22 04 04 13 18}}
Recall that for $H_\theta^{p, \gamma}$ in \eqref{w sobolev} we have $\Sigma^2(L_{\beta}^2) = H_0^{2, \beta}$.
Since $C_c^\infty(\R^d)$ is dense in every $H_\theta^{p,\gamma}$ due to \cite[Proposition 2.2.2]{Lot_2000}, 
it suffices to show that the mapping $m \mapsto \widetilde{m}$ initially defined for $m \in C_c^\infty(\R^d)$ can be bounded from $\Sigma^2\big( L_{\beta+1/2+2\varepsilon}^2\big)$ to $\Sigma^2\big( L_{\beta}^2\big)$.
Since $\|m\|_{\Sigma^2(L_\beta^2)} \lesssim \| m\|_{\Sigma^2(L_{\beta + 1/2 + 2\varepsilon}^2)}$ we consider $\widetilde{m}$ as in \eqref{22 05 11 14 38}.
We actually prove the mapping property of $m\to \widetilde{m}$ for $\beta = 0$, which yields the case of $\beta \in \mathbb{N}$.
For real $\beta$, we interpolate the results of $\beta \in \mathbb{N}$ cases in terms of \eqref{intpltn:Sigma Sobolev}. 
We make use of the following equivalent quantities to $\|m\|_{\Sigma^2(L_n^2)}$, $\|m\|_{\Sigma^2(L_{n+\alpha}^2)}$ for $n\in \mathbb{N}\cup \{0\}$ and $\alpha \in (0,1)$.
\begin{align}
  \|m\|_{\Sigma^2(L_n^2)}^2 
  &\simeq \sum_{\gamma: |\gamma| \leq n} \int_{\R^d} \big| |\xi|^{|\gamma|}\partial_\xi^\gamma m(\xi) \big|^2 \frac{d\xi}{|\xi|^d},\label{besov1}\\
  \|m\|_{\Sigma^2(L_{n+\alpha}^2)}^2
  &\simeq \|m\|_{\Sigma^{2}(L^2_{n})}+\|D^nm\|_{\Sigma^{2}_n(L^2_{\alpha})},\label{besov2}\\
  \| m \|_{\Sigma_{n}^2(L_\alpha^2)}^2
  &\simeq 
    \int_{\R^d} |m(\xi)|^2 |\xi|^{2n-d} d\xi
    + \int_{\R^d} 
      \biggl( \int_{|y-\xi| \leq \frac{|\xi|}{2}} \frac{\big| m(y) - m(\xi) \big|^2}{|y-\xi|^{d+2\alpha}} dy \biggr) 
    \frac{d\xi}{|\xi|^{d-2\alpha -2n}}.\label{besov3}
\end{align}
\eqref{besov1} and \eqref{besov2} are given in Lemma \ref{22.03.10.2},
and \eqref{besov3} is a special case of Proposition \ref{prop:equiv} in Subsection \ref{appendix:equiv}.
Since $m\in \Sigma^2(L_{n+\alpha}^2) \subset \Sigma^2(L_{n}^2)$, 
without loss of any generality we put 
$$
\widetilde{m}(\xi) = \int_0^1 (1-s)^{-\frac{3}{2}-\varepsilon} \bigl(m(\xi) - m(s\xi)\bigr) ds
$$
and let
$$
\int_0^1 (\cdots) = \int_0^{1-\delta} (\cdots) + \int_{1-\delta}^1 (\cdots) = I + II
$$
for some $\delta>0$ chosen later.
Then we consider $I$ and $II$ separately. 
For $I$, we have
\begin{align*}
  \big\| I \big\|_{\Sigma^2(L^2)} 
  \leq \int_0^{1-\delta} (1-s)^{-\frac{3}{2}-\varepsilon} \big( \big\|m \big\|_{\Sigma^2(L^2)} + \big\| m(s\cdot) \big\|_{\Sigma^2(L^2)} \big) ds.
\end{align*}
Since the right hand side of \eqref{besov1} is scaling invariant, it follows that
\begin{align*}
  \big\| I \big\|_{\Sigma^2(L^2)} \leq 2 \delta^{-\frac{3}{2}} \big\|m \big\|_{\Sigma^2(L^2)}.
\end{align*}
Thus it remains to handle $II$-term.
For $II$, we claim that if $s\in (1-\delta, 1)$, then 
\begin{align}\label{2022 0113 1535}
  \big\|m(\cdot) - m(s\cdot) \big\|_{\Sigma^2(L^2)} 
  \lesssim (1-s)^{1/2 + 2\varepsilon} \big\| m \big\|_{\Sigma^2(L_{1/2+2\varepsilon}^2)},
\end{align}
which is given in the next subsection.
By \eqref{2022 0113 1535}, we can show that
\begin{align*}
  \big\| II \big\|_{\Sigma^2(L^2)} 
  \lesssim  \delta^{1 -\varepsilon} \big\| m \big\|_{\Sigma^2(L_{1/2+2\varepsilon}^2)}.
\end{align*}
Thus we have shown that 
\begin{align}\label{ineq:1/2 regular}
  \| \widetilde{m} \|_{\Sigma^2(L^2)} \lesssim \| m \|_{\Sigma^2(L_{1/2 + 2\varepsilon}^2)}.
\end{align}
The inequality \eqref{ineq:1/2 regular} means that $m \in \Sigma^2(L_{1/2 + 2\varepsilon}^2)$ implies $\widetilde{m} \in \Sigma^2(L^2)$.
We want to generalize this statement to the case of $m \in \Sigma^2(L_{n+1/2 + 2\varepsilon}^2)$ and $\widetilde{m} \in \Sigma^2(L_n^2)$.

To do so, note that by  \eqref{besov1} and \eqref{besov2} we have 
\begin{align*}
  \|m\|_{\Sigma^2(L_{n + 1/2 + 2\varepsilon}^2)}
  &\simeq \sum_{|\gamma|\leq n}\big\||\xi|^{|\gamma|} \partial^{\gamma}m(\xi)\big\|_{\Sigma^2(L_{1/2 + 2\varepsilon}^2)},\\
  \|\widetilde{m}\|_{\Sigma^2(L_{n}^2)}
  &\simeq \sum_{|\gamma|\leq n}\big\||\xi|^{|\gamma|} \partial^{\gamma}\widetilde{m}(\xi)\big\|_{\Sigma^2(L^2)}.
\end{align*}
Also, one can observe that $\Big(|\xi|^{|\gamma|} \partial^\gamma m(\xi)\Big)^{\thicksim} = |\xi|^{|\gamma|} \partial^{\gamma}\widetilde{m}(\xi)$.
Then from \eqref{ineq:1/2 regular} we have 
\begin{align*}
  \|\widetilde{m}\|_{\Sigma^2(L_{n}^2)}\,
  &\simeq \sum_{|\gamma|\leq n}\big\||\xi|^{|\gamma|} \partial^{\gamma}\widetilde{m}(\xi)\big\|_{\Sigma^2(L^2)}\\
  &=\sum_{|\gamma|\leq n}\big\|\Big(|\xi|^{|\gamma|} \partial^\gamma m(\xi)\Big)^{\thicksim}\big\|_{\Sigma^2(L^2)}\\
  &\lesssim \sum_{|\gamma|\leq n}\big\||\xi|^{|\gamma|} \partial^\gamma m(\xi)\big\|_{\Sigma^2(L^2_{1/2+\varepsilon})}\\
  &\simeq \|m\|_{\Sigma^2(L_{n + 1/2 + \varepsilon}^2)}.
\end{align*}
Thus we have shown that $m \in \Sigma^2(L_{n+1/2 + \varepsilon}^2)$ implies $\widetilde{m} \in \Sigma^2(L_n^2)$ for $n=0,1,2,\cdots$.

\subsection{Proof of the claim \eqref{2022 0113 1535}}

Recall that $\|m(\cdot) - m(s\cdot)\|_{\Sigma^2(L^2(\R^d))}^2$ is 
\begin{align*}
  \int_{\R^d} \big| m(\xi) - m(s\xi) \big|^2 \frac{d\xi}{|\xi|^d}.
\end{align*}
Put $\delta=\frac{1}{8}$. For a fixed $\xi\in\mathbb{R}^d\setminus\{0\}$, put
\begin{align*}
 B_{1,\xi}&= \{y\in \R^d : |y-\xi| < \delta(1-s)|\xi|\},\\
  B_{2,\xi}&= \{z\in \R^d : |z-s\xi| < \delta (1-s)|s\xi|\}\,,
\end{align*}
and make use of the triangle inequality
$$
|m(\xi)-m(s\xi)|\leq |m(\xi)-m(y)|+|m(y)-m(z)|+|m(z)-m(s\xi)|
$$
to obtain
\begin{align*}
  &\|m(\cdot) - m(s\cdot)\|_{\Sigma^2(L^2)}^2\\
  &\lesssim
    \int_{\R^d}\int_{B_{2,\xi}}\int_{B_{1,\xi}} 
                  \Big(|m(\xi) - m(y)|^2 + |m(y) - m(z)|^2 + |m(z) - m(s\xi)|^2 \Big)
    \frac{dy}{|B_{1,\xi}|} \frac{dz}{|B_{2,\xi}|}\frac{d\xi}{|\xi|^d}\\
  &=:II_1 + II_2 + II_3.
\end{align*}
We deal with $II_1$, $II_2$, and $II_3$, separately.

For $II_1$, since $y\in B_{1,\xi}$ satisfies $|y-\xi| < \delta(1-s)|\xi|$, we have
\begin{align*}
  II_1 
  &= \int_{\R^d} \int_{B_{1,\xi}} |m(\xi) - m(y)|^2 \frac{dy}{|B_{1,\xi}|} \frac{d\xi}{|\xi|^d}\\
  &\simeq \int_{\R^d} \int_{B_{1,\xi}} \frac{|m(\xi) - m(y)|^2}{\big(\delta (1-s)|\xi| \big)^d} dy \frac{d\xi}{|\xi|^d}\\
  &= \delta^{-d} \int_{\R^d} \bigg( \int_{B_{1,\xi}} \frac{|m(\xi) - m(y)|^2}{\big((1-s)|\xi| \big)^{d+1+4\varepsilon}} dy\bigg)
      \big((1-s)|\xi|\big)^{1+4\varepsilon} \frac{d\xi}{|\xi|^d}\\
  &\leq \delta^{1+4\varepsilon} (1-s)^{1+4\varepsilon} \int_{\R^d} \int_{|y-\xi| < \delta|\xi|} \frac{|m(\xi) - m(y)|^2}{|\xi - y|^{d+1+4\varepsilon}} dy \frac{d\xi}{|\xi|^{d-1-4\varepsilon}}.
\end{align*}
Since $\delta<1/2$, it follows that
\begin{align}\label{II_1 estimate}
  II_1 \lesssim (1-s)^{1+4\varepsilon} \|m\|_{\Sigma^2(L_{1/2 + 2\varepsilon})}^2.
\end{align}

For $II_3$, we perform the argument of the case $II_1$ to obtain
\begin{align*}
  II_3
  &= \int_{\R^d} \int_{B_{2,\xi}} |m(z) - m(s\xi)|^2 \frac{dz}{|B_{2,\xi}|} \frac{d\xi}{|\xi|^d}\\
  &\simeq  \int_{\R^d} \int_{B_{2,\xi}} \frac{|m(z) - m(s\xi)|^2}{\bigl(\delta (1-s) |s\xi|\bigr)^d} dz \frac{d\xi}{|\xi|^d}\\
  &\leq \delta^{1+4\varepsilon} (1-s)^{1+4\varepsilon} \int_{\R^d} \int_{|z-\xi|<\delta|\xi|} \frac{|m(z) - m(\xi)|^2}{|z-\xi|^{d+1+4\varepsilon}} dz \frac{d\xi}{|\xi|^{d-1-4\varepsilon}}.
\end{align*}
Again, since $\delta<1/2$
we have
\begin{align}\label{II_3 estimate}
  II_3 \lesssim (1-s)^{1+4\varepsilon} \|m\|_{\Sigma^2(L_{1/2 + 2\varepsilon})}^2.
\end{align}

Lastly, we consider the term $II_2$.
In this case, we need to handle both $y$- and $z$- integrals.
To do this, define a set $E$ for $s\in(1-\delta,1)$.
$$
E := \{(y, z, \xi) \in \R^d\times \R^d\times \R^d : |y-\xi| \leq \delta (1-s)|\xi|, |z-s\xi| \leq \delta (1-s)|s\xi|\}.
$$
Observe that on $E$ we have $|y-\xi| < \delta(1-s)|\xi|\leq \delta^2|\xi|$, which yields
\begin{align*}
  |y|<2|\xi|, \q |\xi| < 2|y|. 
\end{align*}
Moreover, 
\begin{align*}
  |z-y| \leq |z-s\xi|+(1-s)|\xi|+|\xi-y|\leq 2(1-s)|\xi|\leq 2 \delta|\xi|\leq 4\delta|y| , \q s>1-\delta.
\end{align*}
Now we can estimate $II_2$.
\begin{align*}
  II_2 
  &= \int_{\R^d} \int_{B_{2,\xi}} \int_{B_{1,\xi}} |m(y) - m(z)|^2 \frac{dy}{|B_{1,\xi}|} \frac{dz}{|B_{2,\xi}|} \frac{d\xi}{|\xi|^d}\\
  &\simeq  \int\int\int_E \,\,\frac{|m(y) - m(z)|^2}{ (1-s)^{2d} s^d |\xi|^{3d}}\,\, dydz d\xi\\
  &\lesssim 
        (1-s)^{-d+1+4\varepsilon}\int\int\int_E \,\,
              \frac{|m(y) - m(z)|^2}{|y-z|^{d+1+4\varepsilon} |y|^{2d-1-4\varepsilon}} 
                    \,\,dydzd\xi ,
\end{align*}
where we use $|\xi| \simeq |y|$, $|y-z| < 2(1-s)|\xi|$, and $s>1-\delta$ in the last inequality. 
Since $\delta=\frac{1}{8}$, if we apply the change of variables in terms of $E$, 
then $II_2$ is dominated by $\delta^{-2d}(1-\delta)^{-d}$ times of the following quantity:
\begin{align*}
  &(1-s)^{-d+1+4\varepsilon} 
      \int_{y\in \R^d\setminus\{0\}} 
        \biggl( \int_{|z-y| <4\delta |y|} \frac{|m(y) - m(z)|^2}{|y-z|^{d+1+4\varepsilon}} dz \biggr)
        \bigg( \int_{|\xi -y| < 2\delta(1-s)|y|} d\xi \bigg) 
      \frac{dy}{|y|^{2d-1-4\varepsilon}}\\
  &\leq (1-s)^{-d+1+4\varepsilon}
      \int_{y\in \R^d\setminus\{0\}} 
      \biggl( \int_{|z-y| <4\delta |y|} \frac{|m(y) - m(z)|^2}{|y-z|^{d+1+4\varepsilon}} dz \biggr)
      \delta^d (1-s)^d |y|^d \frac{dy}{|y|^{2d-1-4\varepsilon}}\\
  &= \delta^d (1-s)^{1+4\varepsilon} 
      \int_{y\in \R^d\setminus\{0\}} 
      \biggl( \int_{|z-y| <4\delta |y|} \frac{|m(y) - m(z)|^2}{|y-z|^{d+1+4\varepsilon}} dz \biggr)
      \frac{dy}{|y|^{d-1-4\varepsilon}}.
\end{align*}
Since $4\delta=1/2$ and Proposition \ref{prop:equiv}, we have
\begin{align}\label{II_2 estimate}
  II_2 \lesssim (1-s)^{1+4\varepsilon} \|m\|_{\Sigma^2(L_{1/2 + 2\varepsilon})}^2.
\end{align}
Together with \eqref{II_1 estimate}, \eqref{II_3 estimate}, and \eqref{II_2 estimate}, we prove
\begin{align*}
  \big\|m(\cdot) - m(s\cdot) \big\|_{\Sigma^2(L^2)} 
  \lesssim  (1-s)^{\frac{1}{2}+2\varepsilon} \big\| m \big\|_{\Sigma^2 \big(L_{1/2+2\varepsilon}^2\big)},
\end{align*}
which is \eqref{2022 0113 1535}.
This proves the lemma.

\subsection{Equivalent norm for $\Sigma^2(L_s^2)$}\label{appendix:equiv}
  
We show the equivalence of \eqref{besov2}. 
Recall that $\Sigma_\theta^2(\mathcal{B})$ is equipped with the norm given by
\begin{align*}
  \big\| f \big\|_{\Sigma_\theta^2(\mathcal{B})} := \Big(\sum_{j\in\Z} 2^{2j\theta} \big\| f(2^j \cdot) \wh{\psi}(\cdot) \big\|_{\mathcal{B}}^2\Big)^{1/2}.
\end{align*}
For $\mathcal{B} = L_n^2$ or $L_{n+\alpha}^2$ with $\alpha \in(0,1)$, the following equivalences are known:
\begin{lemma}\label{22.03.10.2}\cite[Proposition 2.2.3, Theorem 4.1]{Lot_2000}
  Let $n\in\mathbb{N}_0$ and $\alpha\in(0,1)$\,.
  \begin{enumerate}
    \item $\|m\|_{\Sigma^{2}_{\theta}(L^2_{n})}^2\simeq \sum_{k=0}^n\int_{\Omega}|D^km|^2|x|^{2\theta+2k-d}dx$
    \item $\|m\|_{\Sigma^{2}_{\theta}(L^2_{n+\alpha})}\simeq \|m\|_{\Sigma^{2}_{\theta}(L^2_{n})}+\|D^nm\|_{\Sigma^{2}_{\theta+n}(L^2_{\alpha})}$
  \end{enumerate}
\end{lemma}
Then by making use of Lemma \ref{22.03.10.2}, we can obtain the equivalence \eqref{besov3}.
\begin{proposition}\label{prop:equiv}
  For any $\alpha\in (0,1)$, $p\in (1,\infty)$ and $\theta\in\mathbb{R}$,
  \begin{align*}
    \|m\|_{\Sigma^{2}_{2\theta}(L^2_{\alpha})}^2
    \simeq 
      \int_{\Omega}
        |m|^2|x|^{2\theta-d}
      dx
      +\int_{\Omega}
        \Big(\int_{|y-x|<\frac{|x|}{2}}\frac{|m(x)-m(y)|^2}{|x-y|^{d+2\alpha}}dy\Big)
        |x|^{2\alpha +2\theta-d}
      dx
  \end{align*}
\end{proposition}
\begin{proof}
From the representation of Besov norm on $\mathbb{R}^d$ in \cite[p.189, 190]{Tri1995},
$$
  \|m\|_{L^2_{\alpha}}^2\simeq \|m\|_{L_2}^2+\int_{\mathbb{R}^d} \int_{\mathbb{R}^d}\frac{|m(x)-m(y)|^2}{|x-y|^{d+2\alpha}}dxdy,
$$
we obtain
\begin{align*}
  \|m\|_{\Sigma^{2}_{\theta}(L^2_{\alpha})}^2
  &\simeq 
    \sum_{k\in\mathbb{Z}}2^{2k\theta}
      \|m(2^k\cdot)\wh{\psi}\|_2^2
    +\sum_{k\in\mathbb{Z}}2^{2k\theta}
      \iint_{\mathbb{R}^d\times \mathbb{R}^d}\frac{|m(2^kx)\wh{\psi}(x)-m(2^ky)\wh{\psi}(y)|^2}{|x-y|^{d+2\alpha}}dxdy\\
  &=:I_0+I_1\,.
\end{align*}
By Lemma~\ref{22.03.10.2}, we have
\begin{align}\label{ineq:I0}
  I_0\simeq \int_{\R^d}|m|^2|x|^{2\theta-d}dx\,.
\end{align}

For $I_1$ we perform the change of variables $2^k x \to x$, $2^ky \to y$ so that 
\begin{align*}
  I_1\,
  &:=
    \sum_{k\in\mathbb{Z}}
      \iint_{\mathbb{R}^d\times \mathbb{R}^d}
        \frac{|m(x)\wh{\psi}(2^{-k}x)-m(y)\wh{\psi}(2^{-k}y)|^2}{|x-y|^{d+2\alpha}}2^{k(2\theta-d+2\alpha)}
      dxdy\\
  &=
    \sum_{k\in\mathbb{Z}}
      \iint_{\mathbb{R}^d\times \mathbb{R}^d}
        \frac{\big||x|^{\theta-d/2+\alpha}m(x)\eta(2^{-k}x)-|y|^{\theta-d/2+\alpha}m(y)\eta(2^{-k}y)\big|^2}{|x-y|^{d+2\alpha}}
      dxdy\\
  &\lesssim 
    \sum_{k\in\mathbb{Z}}
      \iint_{|x-y|\geq \frac{|x|}{2}}
        \frac{|\eta(2^{-k}x)|^2|x|^{\theta-d+2\alpha}|m(x)|^2}{|x-y|^{d+2\alpha}}
        +\frac{|\eta(2^{-k}y)|^2|y|^{\theta-d+2\alpha}|m(y)|^2}{|x-y|^{d+2\alpha}}
      dxdy\\
  &\quad 
    +\sum_{k\in\mathbb{Z}}
      \iint_{|x-y|< \frac{|x|}{2}}
        \frac{\big|\eta(2^{-k}x)-\eta(2^{-k}y)\big|^2}{|x-y|^{d+2\alpha}}|x|^{2\theta-d+2\alpha}|m(x)|^2
      dxdy\\
  &\quad
    +\sum_{k\in\mathbb{Z}}
      \iint_{|x-y|< \frac{|x|}{2}}
        |\eta(2^{-k}y)|^2\frac{\big||x|^{\theta-d/2+\alpha}m(x)-|y|^{\theta-d/2+\alpha}m(y)\big|^2}{|x-y|^{d+2\alpha}}
      dxdy\\
  &=:I_{1,1}+I_{1,2}+I_{1,3}
\end{align*}
where $\eta(x)=|x|^{-\theta+d/2-\alpha}\wh{\psi}(x)$.
Note that we apply $|x|^{\theta-d/2+\alpha}  m(x)\,\eta(2^{-k}y)$ for the last inequality when $|x-y| < \frac{|x|}{2}$. 

To progress further, observe that for any $a\in\mathbb{R}$
\begin{align}
  \sum_{k\in\mathbb{Z}}2^{ak}|\eta(2^{-k}x)|^2&\simeq |x|^{a}, \label{22.03.10.1}\\
  \sum_{k\in\mathbb{Z}}2^{ak}|\big(\nabla\eta\big)(2^{-k}x)|^2&\lesssim |x|^{a}\label{22.03.11.1}. 
\end{align}
For $\alpha>0$ we have
\begin{align*}
  \int_{|x-y|\geq \frac{|x|}{3}} |x-y|^{-d-2\alpha} dy &\simeq |x|^{-2\alpha},\\
  \int_{|x-y|\geq \frac{|x|}{2}} |x-y|^{-d-2\alpha} dx &\leq \int_{|x| \geq |y|} |x|^{-d-2\alpha}  dx \simeq |y|^{-2\alpha}.
\end{align*}
Then by \eqref{22.03.10.1} we can estimate $I_{1,1}$.
\begin{align}\label{ineq:I11}
  I_{1,1}
  \lesssim   
    \iint_{|x-y|\geq \frac{|x|}{3}}
      \frac{|x|^{2\theta-d+2\alpha}|m(x)|^2}{|x-y|^{d+2\alpha}}
    dydx
  \lesssim 
    \int_{\mathbb{R}^d}
      |m(x)|^2|x|^{2\theta-d}
    dx\simeq I_0.
\end{align}

For $I_{1,2}$ we make use of \eqref{22.03.11.1} to obtain
\begin{equation}\label{ineq:22 03 11 13 51}
\begin{aligned}
  \sum_k|\eta(2^{-k}x)-\eta(2^{-k}y)|^2
  &\leq \sum_k|x-y|^22^{-2k}\Big(\int_0^1\big|\big(\nabla \eta\big)\big(2^{-k}(1-r)x+2^{-k}ry\big)\big|dr \Big)^2\\
  &\leq |x-y|^2\int_0^1\sum_k2^{-2k}\big|\big(\nabla \eta\big)\big(2^{-k}(1-r)x+2^{-k}ry\big)\big|^2dr\\
  &\lesssim |x-y|^2\int_0^1|(1-r)x+ry|^{-2}dr\\
  &\lesssim |x-y|^2|x|^{-2}
\end{aligned}
\end{equation}
when $|x-y|<|x|/2$. 
By \eqref{ineq:22 03 11 13 51} we have
\begin{align*}
  I_{1,2}\,
  \lesssim 
    \iint_{|x-y|< \frac{|x|}{2}}
      |x-y|^{-d+2(1-\alpha)}|m(x)|^2|x|^{2\theta-d+2(\alpha-1)}
    dxdy.
\end{align*}
Since 
\begin{align*}
  \int_{|x-y|<\frac{|x|}{2}} |x-y|^{-d+2(1-\alpha)} dy \lesssim |x|^{2(1-\alpha)},
\end{align*}
it follows that
\begin{align}\label{ineq:I12}
  I_{1,2}\,\lesssim \int_{\R^d} |m(x)|^2|x|^{2\theta-d}dx\simeq I_0.
\end{align}
Therefore, by \eqref{ineq:I0}, \eqref{ineq:I11}, \eqref{ineq:I12} we have
\begin{align}\label{ineq:I13 0}
  \|m\|_{\Sigma_\theta^2(L^2_{\alpha})}^2
  \lesssim 
    \int_{\R^d} |m(x)|^2|x|^{2\theta-d}dx 
    + I_{1,3}\,.
\end{align}

For $I_{1,3}$ we apply \eqref{22.03.10.1} so that
$$
  I_{1,3}
  \simeq 
    \iint_{|x-y|< \frac{|x|}{2}}
      \frac{\big||x|^{\theta-d/2+\alpha}m(x)-|y|^{\theta-d/2+\alpha}m(y)\big|^2}{|x-y|^{d+2\alpha}}
    dxdy.
$$
Then we have
\begin{equation}\label{ineq:22 03 11 14 30}
\begin{aligned}
  &\iint_{|x-y|< \frac{|x|}{2}}
      \frac{\big||x|^{\theta-d/2+\alpha}m(x)-|y|^{\theta-d/2+\alpha}m(y)\big|^2}{|x-y|^{d+2\alpha}}
   dxdy\\
  \lesssim\,& 
    \iint_{|x-y|< \frac{|x|}{2}}
      \frac{\big||x|^{\theta-d/2+\alpha}m(x)-|x|^{\theta-d/2+\alpha}m(y)\big|^2}{|x-y|^{d+2\alpha}}
    dxdy\\
  &
  + \iint_{|x-y|< \frac{|x|}{2}}
      \frac{\big||x|^{\theta-d/2+\alpha}-|y|^{\theta-d/2+\alpha}\big|^2 |m(y)|^2}{|x-y|^{d+2\alpha}}
    dxdy.
\end{aligned}
\end{equation}
Note that the first term in the RHS of \eqref{ineq:22 03 11 14 30} can be written as
\begin{align}\label{ineq:22 03 11 14 40}
\int_{\R^d} \bigg(\int_{|x-y|<\frac{|x|}{2}} \frac{|m(x)-m(y)|^2}{|x-y|^{d+2\alpha}}dy\bigg) |x|^{2\alpha +2\theta -d} dx,
\end{align}
which is the desired quantity.
For the second term, we make use of the inequality
$$
\big||x|^{a}-|y|^{a}\big|\lesssim_a |x-y||y|^{a-1}
$$
whenever $|x-y|\leq |x|/2$ so that
\begin{align}
  \int_{|x-y|< \frac{|x|}{2}}
    \frac{\big||x|^{\theta-d/2+\alpha}-|y|^{\theta-d/2+\alpha}\big|^2}{|x-y|^{d+2\alpha}} 
  dx
  &\lesssim 
    |y|^{2\theta -d +2\alpha-2} \int_{|x-y|<|y|} |x-y|^{-d-2\alpha +2} dx \nonumber \\
  &\simeq 
    |y|^{2\theta -d +2\alpha-2} \times |y|^{2 - 2\alpha} = |y|^{2\theta -d}.\nonumber
\end{align}
Thus we have
\begin{align}\label{ineq:22 03 11 14 41}
  \iint_{|x-y|< \frac{|x|}{2}}
      \frac{\big||x|^{\theta-d/2+\alpha}-|y|^{\theta-d/2+\alpha}\big|^2 |m(y)|^2}{|x-y|^{d+2\alpha}}
  dxdy
  \lesssim 
    \int_{\R^d} |m(x)|^2 |x|^{2\theta-d} dx.
\end{align}
By \eqref{ineq:I13 0}, \eqref{ineq:22 03 11 14 40}, \eqref{ineq:22 03 11 14 41}, we have
\begin{align}\label{ineq:22 03 11 15 13}
  \|m\|_{\Sigma_\theta^2(L^2_{\alpha})}^2\,
  \lesssim 
    \int_{\R^d}|m(x)|^2|x|^{2\theta-d}dx
    +\int_{\R^d}
      \Big(\int_{|y-x|<\frac{|x|}{2}}\frac{|m(x)-m(y)|^2}{|x-y|^{d+2\alpha}}dy\Big)
      |x|^{2\alpha +2\theta-d}
    dx.
\end{align}

Finally, we show that the converse of \eqref{ineq:22 03 11 15 13} also holds. 
Observe that 
\begin{align*}
  \big| |x|^N m(x) - |x|^N m(y) \big|^2 
  \lesssim
    \big| |x|^N m(x) - |y|^N m(y)\big|^2 
    + \big| |x|^N - |y|^N \big|^2 |m(y)|^2,
\end{align*}
which yields 
\begin{align*}
  &\int_{\R^d}
    \Big(\int_{|y-x|<\frac{|x|}{2}}\frac{|m(x)-m(y)|^2}{|x-y|^{d+2\alpha}}dy\Big)
    |x|^{2\alpha +2\theta-d}
  dx
  \\
&\q \lesssim I_{1,3} +
\int_{\R^d}\Big(\int_{|y-x|<\frac{|x|}{2}}\frac{\big||x|^{\theta-d/2+\alpha}-|y|^{\theta-d/2+\alpha}\big|^2}{|x-y|^{d+2\alpha}}dx\Big)
    |m(y)|^2dy
\\
&\q\lesssim I_{1,3}+I_0,
\end{align*}
where the last inequality holds due to \eqref{ineq:22 03 11 14 41}.
Moreover, one can observe that
\begin{align*}
  &\big| |x|^N m(x) \eta(2^{-k}y) - |y|^N m(y) \eta(2^{-k}y) \big|^2\\
  \lesssim\,&
    \big| |x|^N m(x) \big|^2 |\eta(2^{-k}y) - \eta(2^{-k}x) \big|^2+ \big| |x|^N m(x) \eta(2^{-k}x) - |y|^N m(y) \eta(2^{-k}y) \big|^2,
\end{align*}
which yields
\begin{align*}
  I_{1,3} \lesssim I_{1,2} + I_1.
\end{align*}
Since $I_{1,2} \lesssim I_0$ and $I_0 + I_1 \simeq \| m \|_{\Sigma_\theta^2(L_\alpha^2)}^2$, we have shown that 
\begin{align}\label{ineq:22 03 11 15 28}
  \int_{\R^d}
    \Big(\int_{|y-x|<\frac{|x|}{2}}\frac{|m(x)-m(y)|^2}{|x-y|^{d+2\alpha}}dy\Big)
    |x|^{2\alpha+2\theta-d}
  dx
  \lesssim 
  \| m \|_{\Sigma_\theta^2(L_\alpha^2)}^2.
\end{align}
By \eqref{ineq:22 03 11 15 28}, it follows that 
\begin{align*}
  \int_{\R^d}|m(x)|^2|x|^{2\theta-d}dx
  +\int_{\R^d}
    \Big(\int_{|y-x|<\frac{|x|}{2}}\frac{|m(x)-m(y)|^2}{|x-y|^{d+2\alpha}}dy\Big)
    |x|^{2\alpha +2\theta-d}
  dx
  \lesssim
    \|m\|_{\Sigma_\theta^2(L^2_{\alpha})}^2.
\end{align*}
This proves the proposition.
\end{proof}

\section{Proof of Lemma \ref{modified sqftn} and bilinear interpolation}\label{sec; mod sqftn}

\subsection{Proof of Lemma \ref{modified sqftn}}

We note that Lemma \ref{modified sqftn} is a modified version of the following result in \cite{Ru1986}:
\begin{lemma}[\cite{Ru1986}, Lemma 4]\label{RdF square}
  Let $m$ be a function of class $C^s(\R^d)$ with $s>\frac{d}{2}$ and supported in $\{ 1/2< |\xi| <2\}$,
  and let the $g$-function be given by 
  $$
  G_m(f)(x) = \biggl( \int_0^\infty \big| \big( m(t\xi) \wh{f}(\xi) \big)\check{\,}(x) \big|^2 \frac{dt}{t}\biggr)^{1/2}.
  $$
  Then for $\beta > d/2$ we have
  $$
  \big\| G_m (f) \|_{L^1(\R^d)} 
  \leq  C_\beta \| m \|_{L_\beta^2(\R^d)} \|f\|_{H^1(\R^d)}.
  $$
\end{lemma}

In \cite{Ru1986}, Lemma \ref{RdF square} is proved by showing the following inequality due to the Calder\'on-Zygmund theory:
\begin{align}\label{2022 0110 1307}
  \int_{|x| > 2|y|} \biggl(\int_0^\infty \big|K_t(x-y) - K_t(x)\big|^2 \frac{dt}{t}\biggr)^{1/2} dx 
  \leq C_\beta \|m\|_{L_\beta^2(\R^d)}.
\end{align}
To perform the Calder\'on-Zygmund theory, one needs to obtain initial $L^2$-estimates.
For the $L^2$ estimate, we make use of the Plancherel theorem to obtain
\begin{align}
  \| G_m(f) \|_{L^2}^2 
  &= \int_{\R^d} |\wh{f}(\xi)|^2 \biggl(\int_0^\infty |m(t\xi)|^2 \frac{dt}{t}\biggr) d\xi\nonumber\\
  &\leq \sup_{\xi\in\R^d \setminus \{\vec{0}\}} \| m(t\xi) \|_{\mathcal{H}(\R_+)}^2 \|f\|_{L^2(\R^d)}^2,\label{Gm L2}
\end{align}
where $\mathcal{H}(\R_+)$ denotes the Hilbert space $L^2(\R_+, \frac{dt}{t})$.
Note that \eqref{Gm L2} yields the first assertion of Lemma \ref{modified sqftn}.
Moreover, let $2^k \leq |\xi_0| \leq 2^{k+1}$ for some $k\in\Z$. 
Then we have
\begin{align*}
  \| m(t\xi_0) \|_{\mathcal{H}(\R_+)}^2 
  &\leq \sum_{j\in\Z} \int_{2^j}^{2^{j+1}} |m(t\xi_0)|^2 \frac{dt}{t}\\
  &\leq \sum_{j\in\Z} \sup_{t \in [2^j, 2^{j+1}]}|m(t2^k\xi_0')|^2 \\
  &\leq\sum_{j\in\Z} \sup_{1\leq |\xi'|\leq 4}|m(2^{j+k} \xi')|^2 
  \lesssim  \sum_{j\in\Z}\big\| m(2^j\cdot) \wh{\psi}(\cdot) \big\|_{L^\infty(\R^d)}^2,
\end{align*}
where $\xi_0 = 2^k \xi_0'$ and we apply the change of variable for the second inequality.
Thus it follows that if $m\in \Sigma^2(L^\infty)$, then
\begin{align}\label{vv L2}
  \big\| G_m(f) \big\|_{L^2(\R^d)} 
  \lesssim \| m\|_{\Sigma^2(L^\infty)} \|f\|_{L^2(\R^d)}.
\end{align}

Back to \eqref{2022 0110 1307}, observe that by H\"ormander's multiplier theorem in vector-valued setting, we have
\begin{align*}
  \int_{|x| > 2|y|} \biggl(\int_0^\infty \big|K_t(x-y) - K_t(x)\big|^2 \frac{dt}{t}\biggr)^{1/2} dx 
  &\leq \sup_{j\in\Z} \big\| m(2^jt\xi) \wh{\psi}(\xi) \big\|_{L_\beta^2(\R^d; \mathcal{H}(\R_+))}.
\end{align*}
Then \eqref{2022 0110 1307} can be accomplished by demonstrating the following inequality:
For $k\in \mathbb{N}_0$ and $\mu(t2^j \xi) = m(t2^j \xi) \wh{\psi}(\xi)$
\begin{align}
  \big\| m(t2^j\cdot) \wh{\psi}(\cdot) \big\|_{L_k^2(\R^d; \mathcal{H}(\R_+))}^2 
  &= \sum_{l: |l|\leq k} \int_{\R^d} \| D_\xi^l (\mu(t2^j\xi)) \|_{\mathcal{H}}^2 d\xi\nonumber\\
  &\leq N(k,d) \sum_{l: |l|\leq k} \int_0^\infty \int_{\frac{1}{2}<|\xi|<2} \big|2^{j|l|} t^{|l|} D_\xi^lm(2^j t \xi) \big|^2 d\xi \frac{dt}{t}\nonumber\\
  &= \sum_{l: |l|\leq k} \int_{\R^d} \big| D_\xi^l m(\xi) \big|^2 \bigg(\int_{\frac{|\xi|}{2^{j+1}} < t< \frac{|\xi|}{2^{j-1}}}  (2^j t)^{2|l|-d}  \frac{dt}{t}\bigg) d\xi\nonumber\\
  &= N(k,d) \sum_{l: |l|\leq k} \int_{\R^d} \big| D_\xi^l m(\xi) \big|^2 |\xi|^{2|l|-d} d\xi \nonumber\\
  &\leq N(k,d) \big\|m \big\|_{\Sigma^2(L_k^2)}^2.\label{2022 0110 1647}
\end{align}
In \eqref{2022 0110 1647} we choose $k\in\mathbb{N}_0$, so one can perform interpolation to obtain the following result:
\begin{align*}
  \sup_{j\in\Z}\big\| m(t2^j\xi)\wh{\psi}(\xi) \big\|_{L_\beta^2(\R^d; \mathcal{H}(\R_+))}
  \lesssim \| m\|_{\Sigma^2(L_\beta^2)}, \q\beta\geq 0.
\end{align*}
Thus we conclude that for any $\varepsilon>0$ if $m\in \Sigma^2(L_{d/2+\varepsilon}^2)$, then
  \begin{align}\label{vv L1}
    \big\| G_m(f) \big\|_{L^1(\R^d)} 
    \lesssim \| m \|_{\Sigma^2(L_{d/2+\varepsilon}^2)} \|f\|_{H^1(\R^d)}.
  \end{align}
For $L^\infty \to BMO$ estimates, it is well-known (by duality) that 
$$
  \|G_m\|_{L^\infty \to BMO} \lesssim \|G_m\|_{L^2 \to L^2} + \|G_m\|_{H^1 \to L^1},
$$
which is a special case of \cite[Theorem 4.2]{Ru_Ru_To1986}.
Since $\|G_m\|_{L^2 \to L^2} \lesssim \|m\|_{\Sigma^2(L^\infty)}$ by \eqref{vv L2}, we make use of $\| f\|_{L^\infty(\R^d)} \lesssim \| f\|_{L_s^2(\R^d)}$ for $s>\frac{d}{2}$ and obtain
$$
\|m\|_{\Sigma^2(L^\infty)} \lesssim  \|m\|_{\Sigma^2(L_s^2)}.
$$
Thus together with $\|G_m\|_{H^1 \to L^1} \lesssim \|m\|_{\Sigma^2(L_s^2)}$ for $s>\frac{d}{2}$, we have 
\begin{align}\label{vv BMO}
  \big\| G_m(f) \big\|_{BMO(\R^d)} 
    \lesssim \| m \|_{\Sigma^2(L_{d/2+\varepsilon}^2)} \|f\|_{L^\infty(\R^d)},
\end{align}
which proves Lemma \ref{modified sqftn} together with \eqref{Gm L2} and \eqref{vv L1}.

\subsection{Interpolation argument}\label{subsec intpltn}

We introduce a complex interpolation result for multilinear operators due to Calder\'on \cite{Cal1964}.
\begin{lemma}\label{complex interpolation}
[\cite{Cal1964}, paragraph 10.1]
  Let $L(x_1, \dots, x_n)$ be a multilinear mapping defined for $x_i \in A_i \cap B_i, i=1,\dots, n$ with values in $A\cap B$ and such that
  \begin{align*}
    \big\| L(x_1, x_2, \dots, x_n) \big\|_A \leq M_0 \prod_{i=1}^n \| x_i \|_{A_i},\\
    \big\| L(x_1, x_2, \dots, x_n) \big\|_B \leq M_1 \prod_{i=1}^n \| x_i \|_{B_i}\,.
  \end{align*}
  Then we have
  \begin{align*}
    \big\| L(x_1,x_2, \dots, x_n) \big\|_C \leq M_0^{1-\theta}M_1^\theta  \prod_{i=1}^n \| x_i \|_{C_i},
  \end{align*}
  where $C=[A,B]_{\theta}$, $C_i = [A_i, B_i]_\theta$, $i=1,\dots,n$
and thus $L$ can be extended continuously to a multilinear mapping of $C_1\times \ldots\times C_n$ into $C$\,.
\end{lemma}
To make use of Lemma \ref{complex interpolation}, it suffices to check $x_i\in A_i \cap B_i$ for $i=2, \dots, n$.
In our case, we want $f \in \mathscr{S}^\infty$, where 
$$
\mathscr{S}^\infty:= \{u \in \mathscr{S}(\R^d): \int x^\alpha u(x)dx = 0 \,\,\text{for all multi-indices $\alpha$}\}
$$ 
Then $\mathscr{S}^\infty$ is a dense subset of $H^p$ for $0<p<\infty$ (p.128 in \cite{St2016}) and $f$ can be an element of $L^2 \cap H^1$.
For $(L^2, L^\infty)$-interpolation, we choose $f\in \mathscr{S}(\R^d)$ so that $f$ can be an element of $L^2\cap L^\infty$.

Now we consider $G_{\widetilde{m}}(f)(x)$ as an $L^2(\frac{dt}{t})$-norm of a bilinear operator
$$
B(m,f)(x,t):=T_{\widetilde{m}(t\cdot)}f(x)\,.
$$
The novelty of bilinear approach is that we do not need to construct an analytic family for $m$ to obtain desired interpolated norms of $m$.
Then the strategy is to apply the multilinear interpolation of Lemma \ref{complex interpolation} on the following bilinear estimates given by Lemma \ref{key embedding}:
\begin{align*}
\begin{array}{lll}
  \big\| B(m,f) \big\|_{L^2(\R^d, \mathcal{H})} &\lesssim\sup_{\xi\neq 0}\Big(\int_0^{\infty}|\widetilde{m}(t\xi)|^2\frac{dt}{t}\Big)^{\frac{1}{2}}\|f\|_{L^2(\mathbb{R}^d)}&\lesssim \| m\|_{\Sigma^2(C^{0,1/2+3\varepsilon})} \|f\|_{L^2(\R^d)},\\
  \big\| B(m,f) \big\|_{L^1(\R^d, \mathcal{H})} &\lesssim \| \widetilde{m}\|_{\Sigma^2(L_{d/2+\varepsilon}^2)} \|f\|_{H^1(\R^d)}&\lesssim \|m\|_{\Sigma^2(L_{(d+1)/2+3\varepsilon}^2)} \|f\|_{H^1(\R^d)}.
\end{array}
\end{align*}
Choose $\theta = 2\bigl(\frac{1}{p} - \frac{1}{2}\bigr) \in(0,1)$, then we will show that
\begin{align}
  &\big[L^1\big(\R^d; \mathcal{H}\big), L^2\bigl(\R^d; \mathcal{H}\bigr)\big]_\theta 
        = L^p(\R^d; \mathcal{H}),\label{H interpolation}\\
  &\bigl[\Sigma^2(L_{(d+1)/2+3\varepsilon}^2), \Sigma^2(C^{0,1/2+3\varepsilon})\bigr]_\theta 
  	    = \Sigma^2(B_{p_0}^s),\,\,\,
         p_0 = \frac{2p}{2-p},\, s=d\Big(\frac{1}{p} - \frac{1}{2}\Big)+\frac{1}{2}+3\varepsilon,\label{Besov interpolation}\\
  &\bigl[H^1(\R^d), L^2(\R^d)\bigr]_\theta 
        = L^p(\R^d).\label{hardy interpolation}
\end{align}
Note that the second interpolation follows from \eqref{intpltn:Sigma}.
Now, taking \eqref{H interpolation}, \eqref{Besov interpolation}, \eqref{hardy interpolation} for granted, 
by Lemma \ref{complex interpolation} we can show that
\begin{align*}
  \big\| B(m,f) \big\|_{L^p(\R^d, \mathcal{H})} 
  \lesssim \| m\|_{\Sigma^2(B_{p_0}^s)} \|f\|_{L^p(\R^d)},\q s\geq d\Big(\frac{1}{p} - \frac{1}{2}\Big)+\frac{1}{2}+3\varepsilon.
\end{align*}
When $p\geq2$ we replace $(H^1 \to L^1)$-estimate into $(L^\infty \to BMO)$-estimate, which is \eqref{vv BMO}.
That is, 
\begin{align*}
  &\big[ BMO\bigl(\R^d; \mathcal{H}\bigr), L^2\big(\R^d; \mathcal{H}\big) \big]_\theta 
        = L^p(\R^d; \mathcal{H}),\\
  &\bigl[\Sigma^2(L_{(d+1)/2+3\varepsilon}^2), \Sigma^2(C^{0,1/2+3\varepsilon})\bigr]_\theta 
  	    = \Sigma^2(B_{p_0}^s),\q 
         p_0 = \frac{2p}{p-2}\,,\,\, s=d\Big(\frac{1}{2} - \frac{1}{p}\Big)+\frac{1}{2}+3\varepsilon,\\
  &\bigl[ L^\infty(\R^d), L^2(\R^d) \bigr]_\theta 
        = L^p(\R^d).
\end{align*}
Then by Lemma \ref{complex interpolation} we obtain
\begin{align*}
  \big\| B(m,f) \big\|_{L^p(\R^d, \mathcal{H})} 
  \lesssim \| m\|_{\Sigma^2(B_{p_0}^s)} \|f\|_{L^p(\R^d)},\q s\geq d\Big(\frac{1}{2} - \frac{1}{p}\Big)+\frac{1}{2}+3\varepsilon,
\end{align*}
which yields $\frac{1}{p_0} = |\frac{1}{p} - \frac{1}{2}|$ and $s\geq d|\frac{1}{p} - \frac{1}{2}|+\frac{1}{2}+3\varepsilon$ in general.

For the Hardy space interpolation, it is given in  \cite[Section 5]{Fe_St1972} that $[H^1, L^p]_\theta = [L^1, L^p]_\theta$ for $\theta \in (0,1)$ and $p>1$.
For $BMO(\R^d)$, it is well known in the literature such as \cite[Theorem 3.4.7]{Gr2014} that $[BMO, L^p]_\theta = [L^\infty, L^p]_\theta$ for $\theta \in (0,1)$ and $1\leq p<\infty$, and we refer to \cite{Bla_Xu1991} for a vector-valued analogue.
Therefore we end this section by verifying the vector-valued interpolation and \eqref{Besov interpolation}.

For the interpolation in the vector-valued setting, it is known in \cite[Theorem 1.18.4]{Tri1995} and \cite{Bla_Xu1991} that for an interpolation couple $\{A_0, A_1\}$ and $\theta, \gamma \in (0,1)$
\begin{align*}
[L^1(A_0) , L^{2}(A_1)]_\theta &= L^p\big( [A_0, A_1]_\theta \big),\q \frac{1}{p} = 1-\theta + \frac{\theta}{2},\\
[BMO(A_0), L^2 (A_1) ]_\gamma &= L^q\big( [A_0, A_1]_\gamma \big),\q \frac{1}{q} = \frac{\gamma}{2}.
\end{align*}
In our case, $A_0 = A_1 = \mathcal{H}$ and $\mathcal{H}$ is the Hilbert space. 
Thus it follows that
$$
	[L^1(\mathcal{H}), L^2(\mathcal{H})]_\theta = L^p(\mathcal{H}) \q \text{and}\q [ BMO(\mathcal{H}), L^2(\mathcal{H}) ]_\gamma = L^q(\mathcal{H}).
$$ 
For \eqref{Besov interpolation} together with $\Sigma^2(L_{(d+1)/2+3\varepsilon}^2) = \Sigma^2(B_{2}^{(d+1)/2+3\varepsilon})$, we perform the complex interpolation to obtain
\begin{align*}
\bigl[\Sigma^2(L_{(d+1)/2+3\varepsilon}^2), \Sigma^2(C^{0,1/2 +3\varepsilon}) \bigr]_\theta 
\,&= \bigl[\Sigma^2(B_{2}^{(d+1)/2+3\varepsilon}), \Sigma^2(B_{\infty}^{1/2+ 3\varepsilon})\bigr]_\theta
\\
\,&= \Sigma^2\big( [B_{2}^{(d+1)/2+3\varepsilon}, B_{\infty}^{1/2 +3\varepsilon}]_\theta \big).
\end{align*}
By the interpolation of the Besov space, we have
$$
\bigl[\Sigma^2(L_{(d+1)/2+3\varepsilon}^2), \Sigma^2(C^{0, 1/2+3\varepsilon})\bigr]_\theta 
= \Sigma^2(B_{p_0}^s)
$$
for $p_0 = 2p/2-p$, $s=d(\frac{1}{p} - \frac{1}{2})+\frac{1}{2}+3\varepsilon$.
Interpolation procedure for \eqref{H interpolation} and \eqref{Besov interpolation} is given in Theorem 1.18.1 and Theorem 2.4.1 of \cite{Tri1995} and Theorem 6.4.5 of \cite{Be_Lo2012}.

\section{Applications of Theorem \ref{main}}\label{Applications}

In this section, we provide proofs of Corollaries \ref{2205151153}, \ref{cor: radial}, \ref{cor half wave}, \ref{cor: M LD} and Propositions \ref{prop Lp cvgc}, \ref{prop: PWC}.

\subsection{Proof of Corollary \ref{2205151153} }\label{near zero}
We use the following properties of $B_{p_0}^s$:
\begin{itemize}
\item Since $s>\frac{d}{p_0}+\frac{1}{2}$, $B_{p_0}^s(\mathbb{R}^d)$ is continuously embedded in $C^{0,1/2}(\mathbb{R}^d)=B_{\infty}^{1/2}(\mathbb{R}^d)$.

\item For $f\in B_{p_0}^s(\mathbb{R}^d)$, $\|f\|_{B_{p_0}^s(\mathbb{R}^d)}\simeq \|f\|_{L^{p_0}(\mathbb{R}^d)}+\|f\|_{\dot{B}_{p_0}^s(\mathbb{R}^d)}$, and $\|f(\lambda\,\cdot\,)\|_{\dot{B}_{p_0}^s(\mathbb{R}^d)}\simeq \lambda^{s-\frac{d}{p_0}}\|f\|_{\dot{B}_{p_0}^s(\mathbb{R}^d)}$. 
Here, $\dot{B}_{p_0}^s$ is the homogeneous Besov space whose norm is given by
$$
\|f\|_{\dot{B}_{p_0}^s}=\bigg( \sum_{j \in\mathbb{Z}} \big( 2^{js} \| \psi_j \ast f\|_{L^{p_0}(\R^d)} \big)^{p_0} \bigg)^{1/p_0}\,.
$$
\end{itemize}
The embedding $B_{p_0}^s(\mathbb{R}^d)\subset C^{0,1/2}(\mathbb{R}^d)=B_{\infty}^{1/2}(\mathbb{R}^d)$ implies that
$$
|m(0)|+\frac{|m(\xi)-m(0)|}{|\xi|^{1/2}}\lesssim \|m\|_{B_{p_0}^s}\qquad \text{for any}\,\,\,\,\xi\in\mathbb{R}^d\setminus\{0\}\,.
$$
Let $\phi_0$ supported on $B(0,2^N)$.
Then denote $m_0=m(0)\phi_0$ and $m_1=m\phi_0-m(0)\phi_0$.
Since $\phi_0\in C_c^{\infty}(\mathbb{R}^d)$, we obtain
\begin{align}\label{2205161235}
\big|\mathcal{M}_{m_0}f\big|=|m(0)|\big|\mathcal{M}_{\phi_0}f\big|\lesssim \|m\|_{B_{p_0}^s}|\cM(f)|\,,
\end{align}
where $\cM(f)$ is the Hardy-littlewood maximal function of $f$.

For $\mathcal{M}_{m_1}$, first we consider the case of $p\in(1,\infty)$.
Observe that for $j\leq N$
\begin{align*}
&\|m_1(2^j\,\cdot\,)\widehat{\psi}\|_{L^{p_0}}\lesssim \|\big(m(2^j\,\cdot\,)-m(0)\big)\widehat{\psi}\|_{L^\infty}\lesssim 2^{j/2}\|m\|_{B_{p_0}^s},\q\text{and}\\
&\|m_1(2^j\,\cdot\,)\widehat{\psi}\|_{\dot{B}_{p_0}^s}\lesssim \|m_1(2^j\,\cdot\,)\|_{\dot{B}_{p_0}^s}\\
&\qquad\qquad\qquad \,\,\,\,\,\simeq 2^{j(s-d/p_0)}\|m_1\|_{\dot{B}_{p_0}^s}\lesssim 2^{j/2}\big(\|m\|_{B_{p_0}^s}+|m(0)|\big)\lesssim 2^{j/2}\|m\|_{B_{p_0}^s},
\end{align*}
which yield
\begin{align*}
\|m_1\|_{\Sigma^2(B_{p_0}^s)}^2\,
\simeq \sum_{j\leq N}\Big(\|m_1(2^j\,\cdot\,)\widehat{\psi}\|_{p_0}+\|m_1(2^j\,\cdot\,)\widehat{\psi}\|_{\dot{B}_{p_0}^s}\Big)^2
\lesssim \Big(\sum_{j\leq N}2^j\Big)\|m\|_{B_{p_0}^s}^2.
\end{align*}
Therefore it follows by Theorem~\ref{main} that
\begin{align}\label{2205161236}
\big\|\mathcal{M}_{m_1}f\big\|_{p}\lesssim\|m_1\|_{\Sigma^2(B_{p_0}^s)}\|f\|_p\lesssim \|m\|_{B_{p_0}^s}\|f\|_p\,.
\end{align}
By combining  \eqref{2205161235} and \eqref{2205161236} we obtain that
$$
\big\|\mathcal{M}_mf\big\|_p\leq \big\|\mathcal{M}_{m_0}f\big\|_p+\big\|\mathcal{M}_{m_1}f\big\|_p\lesssim \|m\|_{B_{p_0}^s}\|f\|_p.
$$
The case of $p=1$ and $p=\infty$ are proved by the same argument.

\subsection{Proof of Corollary \ref{cor: radial}}
Due to $\mathcal{M}_m(f) \leq G_{\widetilde{m}}(f)$ in \eqref{2022 0114 sqftn}, it suffices to show 
\begin{align*}
  \big\| G_{\widetilde{m}}(f) \big\|_{L^p(\R^d)} \lesssim \|h\|_{\Sigma^2(L_s^2(\R_+))} \|f\|_{L^p(\R^d)},
\end{align*}
where $s>d\big| \frac{1}{p} - \frac{1}{2} \big| + \frac{1}{2}$ and $m(\xi)=h(|\xi|)$.
By direct calculation, we obtain
\begin{align*}
  \widetilde{m}(\xi)=\widetilde{h}(|\xi|)\quad\text{where}\q \widetilde{h}(r):=h(r)+(1/2+\varepsilon)\int_0^{\infty}(1-s)^{-\frac{3}{2}-\varepsilon}\big(h(r)-h(sr)\big)ds.
\end{align*}
For the $L^2$ boundedness of $G_{\widetilde{m}}$, note that 
\begin{align}\label{22 05 11 14 52}
  \int_0^{\infty}|\widetilde{m}(t\xi)|^2\frac{dt}{t}
  =\int_0^{\infty}|\widetilde{h}(t|\xi|)|^2\frac{dt}{t}
  =\int_0^{\infty}|\widetilde{h}(t)|^2\frac{dt}{t}
  \simeq \|\widetilde{h}\|^2_{\Sigma^2(L^2(\R_+))}.
\end{align}
Then by Lemma \ref{key embedding} we have
\begin{align}\label{22 05 11 14 53}
  \|\widetilde{h}\|^2_{\Sigma^2(L^2(\R_+))} 
  \lesssim \|h\|^2_{\Sigma(L^2_{1/2+2\varepsilon}(\R_+))}
  \lesssim \|h\|^2_{\Sigma(L^2_{1/2+3\varepsilon}(\R_+))}.
\end{align}
Making use of Lemma \ref{22.03.10.2} we also have for a nonnegative integer $k$
\begin{align*}
  \|m\|_{\Sigma^2(L^2_k)}^2
  \simeq \sum_{l\leq k}\int_{\R^d}|m(\xi)|^2|\xi|^{2l-d}d\xi
  \simeq \sum_{l\leq k}\int_0^{\infty}|h(r)|^2|r|^{2l-1}dr
  \simeq \|h\|_{\Sigma^2(L^2_k(\R_+))}^2,
\end{align*}
which yields 
\begin{align*}
  \|m\|_{\Sigma^2(L^2_s)}\lesssim \|h\|_{\Sigma^2(L^2_s(\R_+))}
\end{align*}
for real $s\geq0$.
Thus by the equality of \eqref{Gm L2}, \eqref{22 05 11 14 52} and \eqref{22 05 11 14 53}, we have
\begin{align}\label{ineq:radial L2}
  \big\| G_{\widetilde{m}}(f) \big\|_{L^2(\R^d)} \lesssim \|h\|_{\Sigma^2(L_{1/2 + 3\varepsilon}^2(\R_+))} \|f\|_{L^2(\R^d)},
\end{align}
and by Theorem \ref{main} we have 
\begin{align}\label{ineq:radial L1}
  \big\| G_{\widetilde{m}}(f) \big\|_{L^1(\R^d)} 
  \lesssim \| m \|_{\Sigma^2(L_{\frac{d+1}{2}+3\varepsilon}^2)} \|f\|_{H^1(\R^d)}
  \lesssim \| h\|_{\Sigma^2(L_{\frac{d+1}{2}+3\varepsilon}^2(\R_+))} \|f\|_{H^1(\R^d)}.
\end{align}
Corollary \ref{cor: radial} follows from applying Lemma \ref{complex interpolation} on \eqref{ineq:radial L2}, \eqref{ineq:radial L1} and duality.

\subsection{Proof of Corollary \ref{cor half wave}}\label{pf:half wave}

Before applying Theorem \ref{main}, we recall embedding property of $B_{p,q}^s$.
Since we use $B_{p_0,p_0}^s$ for $p_0\geq2$, $s\in\big(d\bigl|\frac{1}{p} - \frac{1}{2}\bigr| +\frac{1}{2}, \frac{\beta}{\alpha})$,
it follows that 
\begin{align}\label{Besov embedding}
  L_s^{p_0} = F_{p_0, 2}^s \subseteq  F_{p_0, p_0}^s = B_{p_0,p_0}^s.
\end{align}
Thus for $m_{\alpha, \beta}$ it suffices to check
\begin{align*}
  \big\| m_{\alpha, \beta} \big\|_{\Sigma^2(L_s^{p_0}(\R^d))} < \infty.
\end{align*}
Morevoer, $\ell^2$-summation over $\Z$ can be reduced to summation over non-negative integers since $m_{\alpha, \beta}$ vanishes near the origin.
Therefore we have
\begin{align}\label{ineq 220210 1546}
  \big\| m_{\alpha, \beta} \big\|_{\Sigma^2(B_{p_0, p_0}^s(\R^d))}^2 
  \leq \sum_{j\geq N} \big\| m_{\alpha, \beta}(2^j\cdot) \wh{\psi}(\cdot) \big\|_{L_s^{p_0}(\R^d)}^2
\end{align}
for some $N\in\mathbb{Z}$.
To compute $L_s^{p_0}$-norm of $m_{\alpha, \beta}$, recall that $|\partial^\gamma m_{\alpha, \beta}(\xi)| \leq C |\xi|^{-\beta - |\gamma|(1-\alpha)}$,
which yields
\begin{align}\label{ineq:22 05 04 15 56}
  \big\| m_{\alpha, \beta}(2^j\cdot) \wh{\psi}(\cdot) \big\|_{L_s^{p_0}(\R^d)}
  \leq C 2^{js} 2^{-j(\beta + s(1-\alpha))}.
\end{align}
Since we choose $ \beta/\alpha>s > d\bigl|\frac{1}{p} - \frac{1}{2}\bigr| + \frac{1}{2}$, 
the summation in \eqref{ineq 220210 1546} is bounded and we have
\begin{align}\label{ineq 220210 1548}
  \frac{\beta}{\alpha} > d\biggl|\frac{1}{p} - \frac{1}{2}\biggr| + \frac{1}{2}.
\end{align}
Note that \eqref{ineq 220210 1548} is equivalent to
$$
\bigg|\frac{1}{p} -  \frac{1}{2}\bigg|< \frac{2\beta - \alpha}{2\alpha d} = \frac{\beta/\alpha}{d} - \frac{1}{2d},
$$
which proves Corollary \ref{cor half wave}.

\subsection{Proof of Proposition \ref{prop Lp cvgc}}
Note that the symbol of $U_{\alpha,\beta}$ is 
\begin{align}\label{ineq:symbol}
  \frac{e^{it|\xi|^\alpha} -1}{t^\beta}.
\end{align}
We define
$$
	m_{\alpha, \alpha\beta}(t^{1/\alpha}\xi) = \frac{e^{it|\xi|^\alpha} -1}{t^\beta|\xi|^{\alpha\beta}}
$$
in the sense of \eqref{ineq:slow decay}.

Let $0<\beta\leq1$ and $|\xi|\leq 1$.
By Taylor's expansion we have
$$
m_{\alpha, \alpha\beta}(\xi)=\sum_{k=1}^{\infty}\frac{i^k}{k!}|\xi|^{\alpha (k-\beta)}\quad\text{and}\quad m_{\alpha,\alpha}(\xi)=i+\sum_{k=2}^{\infty}\frac{i^k}{k!}|\xi|^{\alpha (k-1)}\,.
$$
For each $k\in\mathbb{N}$
\begin{align*}
\| |2^j \xi |^{\alpha(k-\beta)} \wh{\psi}(\xi) \|_{L_s^{p_0}(\R^d)} 
\lesssim 2^{j\alpha(k-\beta)} \| |\xi|^{\alpha(k-\beta)} \wh{\psi}  \|_{L_s^{p_0}(\R^d)} 
\lesssim 2^{j\alpha(k-\beta)} k^d 8^k.
\end{align*}
Then we have for  $\beta \in (0,1)$ 
\begin{align*}
\big( \sum_{j\leq 0} \big\| m_{\alpha, \alpha\beta}(2^j\cdot) \wh{\psi}(\cdot) \big\|_{L_s^{p_0}(\R^d)}^2 \big)^{1/2}
&\leq \sum_{j\leq 0} \big\| m_{\alpha, \alpha\beta}(2^j\cdot) \wh{\psi}(\cdot) \big\|_{L_s^{p_0}(\R^d)}\\
&\leq \sum_{j\leq 0} \sum_{k=1}^\infty \frac{1}{k!} \| |2^j \xi |^{\alpha(k-\beta)} \wh{\psi}(\xi) \|_{L_s^{p_0}(\R^d)}\\
&\lesssim \sum_{j\leq 0} \sum_{k=1}^\infty \frac{1}{k!}2^{j\alpha(k-\beta)} k^d 8^k\\
&\leq N \sum_{k=1}^\infty \frac{1}{k!}k^d 8^k <\infty .
\end{align*}
For $\beta =1$ since we have $m_{\alpha, \alpha} \varphi_0 = i \varphi_0 + \sum_{k=2}^{\infty}\frac{i^k}{k!}|\xi|^{\alpha (k-1)} \varphi_0$ and $\varphi_0$ is of class $C_c^\infty(\R^d)$, 
we only consider the second term.
\begin{align*}
\bigg( \sum_{j\leq 0} \bigg\| \sum_{k=2}^{\infty}\frac{i^k}{k!}|2^j \xi|^{\alpha (k-1)} \wh{\psi}(\xi) \bigg\|_{L_s^{p_0}(\R^d)}^2 \bigg)^{1/2}
&\leq \sum_{j\leq 0} \sum_{k=2}^\infty \frac{1}{k!} \| |2^j \xi |^{\alpha(k-1)} \wh{\psi}(\xi) \|_{L_s^{p_0}(\R^d)}\\
&\lesssim \sum_{j\leq 0} \sum_{k=2}^\infty \frac{1}{k!}2^{j\alpha(k-1)} k^d 8^k\\
&\leq N \sum_{k=2}^\infty \frac{1}{k!}k^d 8^k <\infty .
\end{align*}

For $|\xi|\geq1$ we have 
$$
	|\partial^\gamma m_{\alpha, \alpha\beta}(\xi)| \lesssim |\xi|^{-\alpha\beta - (1-\alpha)|\gamma|},
$$ 
so $m_{\alpha, \alpha\beta}$ is a slowly decaying Fourier multiplier. 
Thus we make use of \eqref{Besov embedding}, \eqref{ineq 220210 1546}, \eqref{ineq:22 05 04 15 56} to obtain
$$
\big\| m_{\alpha, \alpha\beta}(2^j\cdot) \wh{\psi}(\cdot) \big\|_{L_s^{p_0}(\R^d)}
  \leq C 2^{js} 2^{-j(\alpha\beta + s(1-\alpha))},
$$
which yields $\beta>d|\frac{1}{p} - \frac{1}{2}| +\frac{1}{2}$ by Theorem \ref{main}.
Then it follows that 
\begin{alignat}{2}\label{ineq:22 03 02 17 21}
  \big\| \sup_{t>0}\big| U_{\alpha,\beta}f(\cdot, t) \big| \big\|_{L^p(\R^d)}
& =\big\| \sup_{t>0}\big| U_{\alpha,\beta}f(\cdot, t^{\alpha})\big| \big\|_{L^p(\R^d)}
   &&\nonumber\\
&=\big\| \sup_{t>0}\big| T_{m_{\alpha,\alpha\beta}(t\,\cdot\,)}\big((-\Delta)^{\alpha\beta}f\big) \big| \big\|_{L^p(\R^d)}
   &&\lesssim \| f\|_{\dot{L}_{\alpha\beta}^p(\R^d)},
\end{alignat}
whenever $|\frac{1}{p} - \frac{1}{2}|< \frac{\beta}{d} - \frac{1}{2d}$.
Together with $\beta\leq1$, we have $\frac{d -2\beta+1}{2d}<\frac{1}{p} < \frac{d+ 2\beta-1}{2d}$.
Finally, \eqref{ineq:22 03 02 17 21} yields that 
\begin{align*}
  \lim_{t\to 0} \big\|  U_{\alpha,\beta}f(\cdot, t) \big\|_{L^p(\R^d)} = \big\| \lim_{t\to 0}  U_{\alpha,\beta}f(\cdot, t) \big\|_{L^p(\R^d)}=0,
\end{align*}
since $\lim_{t\to  0} \frac{e^{itA}-1}{t^\beta} =0$ for a constant $A$ and $\beta\in(0,1)$.
This proves the proposition.

\subsection{Proof of Corollary \ref{cor: M LD}}

Since we assume that $m_{a,b}\in L^2_{(d+1)/2+\varepsilon,loc}$ for some $\varepsilon>0$, it follows that $m\phi \in B_2^s(\R^d)$ for any $\phi \in C_c^\infty(\R^d)$.
Thus by Corollary \ref{2205151153}
 it suffices to check 
\begin{align}\label{ineq:22 04 18 19 51}
  \sum_{j\geq0} \|m_{a,b}(2^j\cdot) \wh{\psi}(\cdot) \|_{B_{p_0}^s(\R^d)}^2 <\infty.
\end{align}
To show \eqref{ineq:22 04 18 19 51}, we make use of \eqref{Besov embedding} and \eqref{condi: M LD} to obtain
\begin{align}\label{ineq:22 04 18 19 57}
  \sum_{j\geq0} \|m_{a,b}(2^j\cdot) \wh{\psi}(\cdot) \|_{B_{p_0}^s(\R^d)}^2
  \leq \sum_{j\geq0} \|m_{a,b}(2^j\cdot) \wh{\psi}(\cdot) \|_{L_s^{p_0}(\R^d)}^2
  \lesssim \sum_{j\geq0} 2^{-j 2\min(a, b-s)}.
\end{align}
Since we choose $a,b>0$ and $s>d\big| \frac{1}{p} - \frac{1}{2} \big| + \frac{1}{2}$, the RHS of \eqref{ineq:22 04 18 19 57} converges if $b-s >0$. 
That is, we have
\begin{align*}
  b>s>d\big| \frac{1}{p} - \frac{1}{2} \big| + \frac{1}{2},
\end{align*}
which yields $\frac{d+1 -2b}{2d} < \frac{1}{p} < \frac{d-1 +2b}{2d}$.
This proves the corollary.

\subsection{Proof of Proposition \ref{prop: PWC}}

By the Fourier transform, we know that 
\begin{align*}
  \frac{f(x) - T_{m(t\cdot)}f(x)}{t^\alpha} = \int_{\R^d} e^{i\langle x, \xi \rangle} \frac{1 - m(t\xi)}{|t\xi|^\alpha} |\xi|^\alpha \wh{f}(\xi) d\xi.
\end{align*}
We define $m_{\alpha}(\xi):= \frac{1 - m(\xi)}{|\xi|^\alpha}$.
Since $|1-m(\xi)| \lesssim |\xi|$ due to $m(0)=1$ and the mean value theorem, $m_\alpha$ is well defined and $|m_\alpha(\xi)| \lesssim (1+|\xi|)^{-\alpha}$.
Also, due to $|\partial^\gamma m (\xi) | \lesssim (1+ |\xi|)^{-\beta}$ and $|1-m(\xi)| \lesssim |\xi|$, it follows that for $|\xi|\leq1$
\begin{align*}
	|\partial^\gamma m_\alpha (\xi) | &\lesssim |\xi|^{-\alpha-|\gamma|+1}
\end{align*}
Then for $|\xi|\leq 1$, we have for any $s\geq0$
\begin{align*}
   \sum_{j\leq 0} \|m_\alpha(2^j\cdot) \wh{\psi}(\cdot) \|_{L_s^{p_0}(\R^d)}^2
   \lesssim \sum_{j\leq 0}2^{-2j(\alpha-1)} <\infty.
\end{align*}
For $|\xi|\geq1$, it follows that for $1/p_0 = |1/p -1/2|$ and $s>d|1/p - 1/2|+1/2$ as in \eqref{condi: M LD}
\begin{align*}
  \| m_\alpha(2^j \cdot) \wh{\psi}(\cdot) \|_{L_s^{p_0}(\R^d)} 
  &\lesssim 2^{-j\alpha} (1 + \| m(2^j \cdot) \wh{\psi}(\cdot) \|_{L_s^{p_0}})\\
  &\lesssim 2^{-j\alpha} + 2^{-j\alpha}2^{-j(\beta -s)} \leq   2^{-j \min( \alpha, \alpha+\beta -s)}.
\end{align*}
Thus we can conclude that $m_\alpha$ is a Fourier multiplier of type $m_{\alpha, \alpha+\beta}$ given in \eqref{condi: M LD}.
Therefore, Proposition \ref{prop: PWC} can be proved by Corollary \ref{cor: M LD}.

\section*{Acknowledgement}

The authors are sincerely grateful to the refree for his/her careful reading and valuable comments.
We also thank to Jae-Hwan Choi and Kwan Woo for useful conversation.

\nocite{*}


\end{document}